\documentclass[reqno, 12pt]{amsart}

\makeatletter
\let\origsection=\section \def\section{\@ifstar{\origsection*}{\mysection}} 
\def\mysection{\@startsection{section}{1}\z@{.7\linespacing\@plus\linespacing}{.5\linespacing}{\normalfont\scshape\centering\S}}
\makeatother    

\usepackage{geometry}
\numberwithin{equation}{section}
\numberwithin{figure}{section}

\usepackage{xcolor}
\usepackage{hyperref}
\hypersetup{
    unicode,
    colorlinks,
    linkcolor={red!60!black},
    citecolor={green!60!black},
    urlcolor={blue!60!black}
}

\usepackage{amsfonts,amsthm,amssymb,amsmath}
\usepackage[T1]{fontenc}
\usepackage{microtype}
\usepackage{graphicx}
\usepackage{tikz}
\usepackage{tkz-berge}
\usepackage[usenames,dvipsnames]{pstricks}
\usepackage{epsfig}
\usepackage{enumitem}
\setlist[enumerate]{label=(\arabic*), ref=(\arabic*)}
\usepackage{verbatim}

\newtheorem*{ubqconjecture}{The Ubiquity Conjecture}

\newtheorem{theorem}{Theorem}[section]
\newtheorem{lemma}[theorem]{Lemma}

\newtheorem{corollary}[theorem]{Corollary}

\theoremstyle{definition}
\newtheorem{definition}[theorem]{Definition}

\newcommand{\N}{\mathbb{N}}

\newcommand{\Nbb}{\mathbb{N}}
\newcommand{\Pcal}{{\mathcal P}}

\newcommand{\Rcal}{{\mathcal R}}
\newcommand{\Scal}{{\mathcal S}}
\newcommand{\Tcal}{{\mathcal T}}
\newcommand{\Fcal}{{\mathcal F}}
\renewcommand{\triangleleft}{\vartriangleleft}
\renewcommand{\leq}{\leqslant}
\renewcommand{\preceq}{\preccurlyeq}
\renewcommand{\rho}{\varrho}
\newcommand{\nottriangleleft}{\not\kern-1pt\mathrel{\triangleleft}}

\DeclareMathOperator{\cf}{cf}

\newcommand{\Gtribe}{{$G$-\text{tribe}}}
\newcommand{\Gsubtribe}{{$G$-\text{subtribe}}}

\author[Bowler, Elbracht, Erde, Gollin, Heuer, Pitz, Teegen]{Nathan Bowler \and Christian Elbracht \and Joshua Erde \and Pascal Gollin \and Karl Heuer \and Max Pitz \and Maximilian Teegen}
\title[Topological ubiquity of trees]{Ubiquity in graphs I: Topological ubiquity of trees}



\begin{document}
\begin{abstract}
    Let $\triangleleft$ be a relation between graphs. We say a graph $G$ is \emph{$\triangleleft$-ubiquitous} if whenever $\Gamma$ is a graph with $nG \triangleleft \Gamma$ for all $n \in \mathbb{N}$, then one also has $\aleph_0 G \triangleleft \Gamma$, where $\alpha G$ is the disjoint union of $\alpha$ many copies of $G$.

    The \emph{Ubiquity Conjecture} of Andreae, a well-known open problem in the theory of infinite graphs, asserts that every locally finite connected graph is ubiquitous with respect to the minor relation.

    In this paper, which is the first of a series of papers making progress towards the Ubiquity Conjecture, we show that all trees are ubiquitous with respect to the topological minor relation, irrespective of their cardinality. 
    This answers a question of Andreae from 1979.
\end{abstract}
\maketitle
\date{}

\section{Introduction}
Let $\triangleleft$ be a relation between graphs, for example the subgraph relation $\subseteq$, the topological minor relation $\leq$ or the minor relation $\preceq$. 
We say that a graph $G$ is \emph{$\triangleleft$-ubiquitous} if whenever $\Gamma$ is a graph with $nG \triangleleft \Gamma$ for all $n \in \mathbb{N}$, 
then one also has $\aleph_0 G \triangleleft \Gamma$, where $\alpha G$ is the disjoint union of $\alpha$ many copies of $G$. 

Two classic results of Halin \cite{H65,H70} say that both the ray and the double ray are $\subseteq$-ubiquitous, i.e.\ any graph which contains arbitrarily large collections of disjoint (double) rays must contain an infinite collection of disjoint (double) rays. 
However, even quite simple graphs can fail to be $\subseteq$ or $\leq$-ubiquitous, see e.g.\ \cite{A77,W76,L76}, examples of which, due to Andreae \cite{A13}, are depicted in Figures \ref{f:subgraph} and \ref{f:topsubgraph} below.

\begin{figure}[ht!]
    \center
    \begin{tikzpicture}[scale=1.3]
        \node[draw, circle,scale=.3, fill] () at (0,0) {};
        \node[draw, circle,scale=.3, fill] () at (0,1) {};
        \node[draw, circle,scale=.3, fill] () at (1,0) {};
        \node[draw, circle,scale=.3, fill] () at (1,1) {};
        \node[draw, circle,scale=.3, fill] () at (2,0) {};
        \node[draw, circle,scale=.3, fill] () at (2,1) {};
        \node[draw, circle,scale=.3, fill] () at (3,0) {};
        \node[draw, circle,scale=.3, fill] () at (3,1) {};
        \node[draw, circle,scale=.3, fill] () at (4,0) {};
        \node[draw, circle,scale=.3, fill] () at (4,1) {};
        \node[draw, circle,scale=.3, fill] () at (5,0) {};
        \node[draw, circle,scale=.3, fill] () at (5,1) {};
        \draw (0,0)--(5,0) (0,0)--(0,1) (1,0)--(1,1) (2,0)--(2,1) (3,0)--(3,1) (4,0)--(4,1) (5,0)--(5,1) (5,0)--(5.5,0);
        \node at (6,.5) {$\ldots$};
    \end{tikzpicture}
    \caption{A graph which is not $\subseteq$-ubiquitous.}
    \label{f:subgraph}
\end{figure}
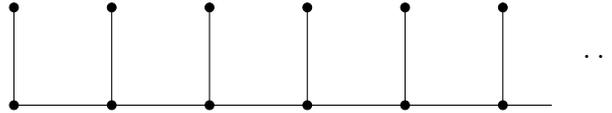

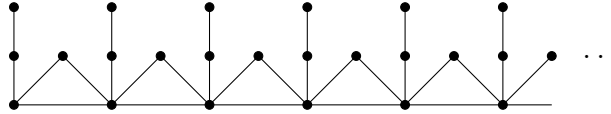
\begin{figure}[ht!]
    \center
    \begin{tikzpicture}[scale=1.3]
        \node[draw, circle,scale=.3, fill] () at (0,0) {};
        \node[draw, circle,scale=.3, fill] () at (0,1) {};
        \node[draw, circle,scale=.3, fill] () at (1,0) {};
        \node[draw, circle,scale=.3, fill] () at (1,1) {};
        \node[draw, circle,scale=.3, fill] () at (2,0) {};
        \node[draw, circle,scale=.3, fill] () at (2,1) {};
        \node[draw, circle,scale=.3, fill] () at (3,0) {};
        \node[draw, circle,scale=.3, fill] () at (3,1) {};
        \node[draw, circle,scale=.3, fill] () at (4,0) {};
        \node[draw, circle,scale=.3, fill] () at (4,1) {};
        \node[draw, circle,scale=.3, fill] () at (5,0) {};
        \node[draw, circle,scale=.3, fill] () at (5,1) {};
        \node[draw, circle,scale=.3, fill] () at (.5,.5) {};
        \node[draw, circle,scale=.3, fill] () at (1.5,.5) {};
        \node[draw, circle,scale=.3, fill] () at (2.5,.5) {};
        \node[draw, circle,scale=.3, fill] () at (3.5,.5) {};
        \node[draw, circle,scale=.3, fill] () at (4.5,.5) {};
        \node[draw, circle,scale=.3, fill] () at (5.5,.5) {};
        \node[draw, circle,scale=.3, fill] () at (0,.5) {};
        \node[draw, circle,scale=.3, fill] () at (1,.5) {};
        \node[draw, circle,scale=.3, fill] () at (2,.5) {};
        \node[draw, circle,scale=.3, fill] () at (3,.5) {};
        \node[draw, circle,scale=.3, fill] () at (4,.5) {};
        \node[draw, circle,scale=.3, fill] () at (5,.5) {};
        \draw (0,0)--(5,0) (0,0)--(0,1) (1,0)--(1,1) (2,0)--(2,1) (3,0)--(3,1) (4,0)--(4,1) (5,0)--(5,1) (5,0)--(5.5,0) (5,0)--(5.5,0.5)
(0,0)--(.5,.5)--(1,0) (1,0)--(1.5,.5)--(2,0)
(2,0)--(2.5,.5)--(3,0) (3,0)--(3.5,.5)--(4,0)
(4,0)--(4.5,.5)--(5,0)
;
        \node at (6,.5) {$\ldots$};
    \end{tikzpicture}
    \caption{A graph which is not $\leq$-ubiquitous.}
    \label{f:topsubgraph}
\end{figure}

However, for the minor relation, no such simple examples of non-ubiquitous graphs are known. Indeed, one of the most important problems in the theory of infinite graphs is the so-called \emph{Ubiquity Conjecture} due to Andreae \cite{A02}. 

\begin{ubqconjecture}
    Every locally finite connected graph is $\preceq$-ubiquitous.
\end{ubqconjecture}

In \cite{A02}, Andreae constructed a graph that is not $\preceq$-ubiquitous. However, this construction relys on the existence of a counterexample to the well-quasi-ordering of infinite graphs under the minor relation, for which counterexamples are only known with very large cardinality \cite{T88}. 
In particular, it is still an open question whether or not there exists a countable connected graph which is not $\preceq$-ubiquitous.

In his most recent paper on ubiquity to date, Andreae \cite{A13} exhibited infinite families of locally finite graphs for which the ubiquity conjecture holds. The present paper is the first in a series of papers \cite{BEEGHPTII,BEEGHPTIII,BEEGHPTIV} making further progress towards the ubiquity conjecture, with the aim being to show that all graphs of bounded tree-width are ubiquitous. 

As a first step towards this, we in particular need to deal with infinite trees, for which one even gets affirmative results regarding ubiquity under the topological minor relation. Halin showed in \cite{halin1975problem} that all trees of maximum degree $3$ are $\leq$-ubiquitous. Andreae improved this result to show that all \emph{locally finite} trees are $\leq$-ubiquitous \cite{A79}, and asked if his result could be extended to arbitrary trees \cite[p.~214]{A79}. 
Our main result of this paper answers this question in the affirmative.

\begin{theorem}\label{t:tree}
    Every tree is ubiquitous with respect to the topological minor relation.
\end{theorem}

The proof will use some results about the well-quasi-ordering of trees under the topological minor relation of Nash-Williams \cite{nash1965well} and Laver \cite{laver1978better}, as well as some notions about the topological structure of infinite graphs \cite{Ends}. Interestingly, most of the work in proving Theorem \ref{t:tree} lies in dealing with the countable case, where several new ideas are needed. In fact, we will prove a slightly stronger statement in the countable case, which will allow us to derive the general result via transfinite induction on the cardinality of the tree, using some ideas from Shelah's singular compactness theorem \cite{shelah1975compactness}.

To explain our strategy, let us fix some notation. 
When $H$ is a subdivision of $G$ we write $G \leq^{*} H$. Then, $G \leq \Gamma$ means that there is a subgraph $H \subseteq \Gamma$ which is a subdivision of $G$, that is, $G \leq^* H$. If $H$ is a subdivision of $G$ and $v$ a vertex of $G$, then we denote by $H(v)$ the corresponding vertex in $H$. More generally, given a subgraph $G' \subseteq G$, we denote by $H(G')$ the corresponding subdivision of $G'$ in $H$.

Now, suppose we have a rooted tree $T$ and a graph $\Gamma$. 
Given a vertex $t \in T$, let $T_t$ denote the subtree of $T$ rooted in $t$.
We say that a vertex $v \in \Gamma$ is {\em $t$-suitable} if there is some subdivision $H$ of $T_t$ in $\Gamma$ with $H(t) = v$. 
For a subtree $S \subseteq T$ we say that a subdivision $H$ of $S$ in $\Gamma$ is {\em $T$-suitable} if for each vertex $s\in V(S)$ the vertex $H(s)$ is $s$-suitable, 
i.e.\ for every $s \in V(S)$ there is a subdivision $H'$ of $T_s$ such that $H'(s) = H(s)$.

An {\em $S$-horde} is a sequence $(H_i \colon i \in \N)$ of disjoint 
suitable subdivisions of $S$ in $\Gamma$. 
If $S'$ is a subtree of $S$, then we say that an $S$-horde $(H_i \colon i \in \N)$ {\em extends} an $S'$-horde $(H'_i \colon i \in \N)$ if for every $i \in \mathbb{N}$ we have $H_i(S') = H_i'$.

In order to show that an arbitrary tree $T$ is $\leq$-ubiquitous, our rough strategy will be to build, by transfinite recursion, $S$-hordes for larger and larger subtrees $S$ of $T$, each extending all the previous ones, until we have built a $T$-horde. 
However, to start the induction it will be necessary to show that we can build $S$-hordes for countable subtrees $S$ of $T$. This will be done in the following key result of this paper:

\begin{theorem}\label{t:countembed}
    Let $T$ be a tree, $S$ a countable subtree of $T$ and $\Gamma$ a graph such that $nT \leq \Gamma$ for every $n \in \Nbb$. 
    Then there is an $S$-horde in $\Gamma$.
\end{theorem}

Note that Theorem \ref{t:countembed} in particular implies $\leq$-ubiquity of countable trees. 

We remark that whilst the relation $\preceq$ is a relaxation of the relation $\leq$, which is itself a relaxation of the relation $\subseteq$, it is not clear whether $\subseteq$-ubiquity implies $\leq$-ubiquity, or whether $\leq$-ubiquity implies $\preceq$-ubiquity. 
In the case of Theorem \ref{t:tree} however, it is true that arbitrary trees are also $\preceq$-ubiquitous, although the proof involves some extra technical difficulties that we will deal with in a later paper \cite{BEEGHPTIV}. 
We note, however, that it is surprisingly easy to show that countable trees are $\preceq$-ubiquitous, since it can be derived relatively straightforwardly from Halin's grid theorem, see \cite[Theorem 1.7]{BEEGHPTII}. 

This paper is structured as follows: In Section~\ref{sec2}, we provide background on rooted trees, rooted topological embeddings of rooted trees (in the sense of Kruskal and Nash-Williams), and ends of graphs. 
In our graph theoretic notation we generally follow the textbook of Diestel \cite{D16}. 
Next, Sections~\ref{sec3} to \ref{secGtribes} introduce the key ingredients for our main ubiquity result. 
Section~\ref{sec3}, extending ideas from Andreae's \cite{A79}, lists three useful corollaries of Nash-Williams' and Laver's result that (labelled) trees are well-quasi-ordered under the topological minor relation, 
Section~\ref{sec4} investigates under which conditions a given family of disjoint rays can be rerouted onto another family of disjoint  rays, 
and Section~\ref{secGtribes} shows that without loss of generality, we already have quite a lot of information about how exactly our copies of $nG$ are placed in the host graph $\Gamma$.

Using these ingredients, we give a proof of the countable case, i.e.\ of Theorem~\ref{t:countembed}, in Section~\ref{sec:countable-subtrees}. 
Finally, Section~\ref{sec:Induction} contains the induction argument establishing our main result, Theorem~\ref{t:tree}.

\section{Preliminaries}
\label{sec2}

\begin{definition}
    A \emph{rooted graph} is a pair $(G,v)$ where $G$ is a graph and $v \in V(G)$ is a vertex of $G$ which we call the \emph{root}. 
    Often, when it is clear from the context which vertex is the root of the graph, we will refer to a rooted graph $(G,v)$ as simply $G$.

    Given a rooted tree $(T,v)$, we define a partial order $\le$, which we call the \emph{tree-order}, on $V(T)$ by letting $x \le y$ if the unique path between $y$ and $v$ in $T$ passes through $x$. See \cite[Section~1.5]{D16} for more background. 
    For any edge $e \in E(T)$ we denote by $e^-$ the endpoint closer to the root and by $e^+$ the endpoint further from the root. 
    For any vertex $t$ we denote by $N^+(t)$ the set of \emph{children of $t$} in $T$, the neighbours $s$ of $t$ satisfying $t\le s$. 
    The subtree of $T$ rooted at $t$ is denoted by $(T_t,t)$, that is, the induced subgraph of $T$ on the set of vertices $\{ s \in V(T) \colon t \le s\}$. 

    We say that a rooted tree $(S,w)$ is a \emph{rooted subtree} of a rooted tree $(T,v)$ if $S$ is a subgraph of $T$ such that the tree order on $(S,w)$ agrees with the induced tree order from $(T,v)$. 
In this case we write $(S,w) \subseteq_r (T,v)$.

    We say that a rooted tree $(S,w)$ is a \emph{rooted topological minor} of a rooted tree $(T,v)$ if there is a subgraph $S'$ of $T$ which is a subdivision of $S$ such that for any $x \le y \in V(S)$, $S'(x) \le S'(y)$ in the tree-order on $T$. 
    We call such an $S'$ a \emph{rooted subdivision of $S$}. 
    In this case we write $(S,w) \leq_r (T,v)$, cf.~\cite[Section~12.2]{D16}.
\end{definition}

\begin{definition}[{Ends of a graph, cf.~\cite[Chapter~8]{D16}}]
    An \emph{end} in an infinite graph $\Gamma$ is an equivalence class of rays, where two rays $R$ and $S$ are equivalent if and only if there are infinitely many vertex disjoint paths between $R$ and $S$ in $\Gamma$. We denote by $\Omega(\Gamma)$ the set of ends in $\Gamma$. 
    Given any end $\epsilon \in \Omega(\Gamma)$ and a finite set $X \subseteq V(\Gamma)$ there is a unique component of $\Gamma - X$ which contains a tail of every ray in $\epsilon$, which we denote by $C(X,\epsilon)$.

    A vertex $v \in V(\Gamma)$ \emph{dominates} an end $\omega$ if there is a ray $R \in \omega$ such that there are infinitely many vertex disjoint $v$\,--\,$R$\,-paths in $\Gamma$. 
\end{definition}

\begin{definition}
    For a path or ray $P$ and vertices $v,w \in V(P)$, let $vPw$ denote the subpath of $P$ with endvertices $v$ and $w$. If $P$ is a ray, let $Pv$ denote the finite subpath of $P$ between the initial vertex of $P$ and $v$, and let $vP$ denote the subray (or \emph{tail}) of $P$ with initial vertex $v$.
    
    Given two paths or rays $P$ and $Q$ which are disjoint but for one of their endvertices, we write $PQ$ for the \emph{concatenation of $P$ and $Q$}, that is the path, ray or double ray $P \cup Q$. Since concatenation of paths is associative, we will not use parentheses. Moreover, if we concatenate paths of the form $vPw$ and $wQx$, then we omit writing $w$ twice and denote the concatenation by $vPwQx$.
\end{definition}

\section{\texorpdfstring{Well-quasi-orders and $\kappa$-embeddability}%
{Well-quasi-orders and \textkappa-embeddability}}
\label{sec3}

\begin{definition}
    Let $X$ be a set and let $\triangleleft$ be a binary relation on $X$. 
    Given an infinite cardinal $\kappa$ we say that an element $x \in X$ is \emph{$\kappa$-embeddable \textup{(}with respect to $\triangleleft$\textup{)} in $X$} if there are at least $\kappa$ many elements $x' \in X$ such that $x \triangleleft x'$.
\end{definition}

\begin{definition}[well-quasi-order]
    A binary relation $\triangleleft$ on a set $X$ is a \emph{well-quasi-order} if it is reflexive and transitive, and for every sequence $x_1,x_2,\ldots \in X$ there is some $i < j$ such that $x_i \triangleleft x_j$. 
\end{definition}

\begin{lemma}
    \label{lem_rotateKappaCollections}
    Let $X$ be a set and let $\triangleleft$ be a well-quasi-order on $X$. 
For any infinite cardinal $\kappa$ the number of elements of $X$ which are not $\kappa$-embeddable with respect to $\triangleleft$ in $X$ is less than $\kappa$.
\end{lemma}

\begin{proof}
    For $x \in X$ let $U_x= \{ y \in X \colon x \triangleleft y\}$. 
    Now suppose for a contradiction that the set $A \subseteq X$ of elements which are not $\kappa$-embeddable with respect to $\triangleleft$ in $X$ has size at least $\kappa$. 
    Then, we can recursively pick a sequence $(x_n \in A)_{n \in \Nbb}$ such that $x_m \nottriangleleft x_n$ for $m < n$. 
    Indeed, having chosen all $x_m$ with $m < n$ it suffices to choose $x_n$ to be any element of the set $A \setminus \bigcup_{m < n} U_{x_m}$, which is nonempty since $A$ has size $\kappa$ but each $U_{x_m}$ has size $< \kappa$.

    By construction  we have $x_{m} \nottriangleleft x_n$ for $m<n$, contradicting the assumption that $\triangleleft$ is a well-quasi-order on $X$.
\end{proof}

We will use the following theorem of Nash-Williams on well-quasi-ordering of rooted trees, and its extension by Laver to labelled rooted trees.

\begin{theorem}[Nash-Williams \cite{nash1965well}]\label{t:NW}
    The relation $\leq_r$ is a well-quasi order on the set of rooted trees.
\end{theorem}

\begin{theorem}[Laver \cite{laver1978better}]\label{t:Laver}
    The relation $\leq_r$ is a well-quasi order on the set of rooted trees with finitely many labels, 
    i.e.\ for every finite number $k \in \N$, whenever $(T_1,c_1),(T_2,c_2),\ldots$ is a sequence of rooted trees with $k$-colourings $c_i \colon T_i \to [k]$, there is some $i < j$ such that there exists a subdivision $H$ of $T_i$ with $H \subseteq_r T_j$ and $c_i(t) = c_j(H(t))$ for all $t \in T_i$.\footnote{In fact, Laver showed that rooted trees labelled by a \emph{better-quasi-order} are again better-quasi-ordered under $\leq_r$ respecting the labelling, but we shall not need this stronger result.}
\end{theorem}


Together with Lemma~\ref{lem_rotateKappaCollections} these results give us the following three corollaries:

\begin{definition}
    \label{def_kappachildren}
    Let $(T,v)$ be an infinite rooted tree. 
    For any vertex $t$ of $T$ and any infinite cardinal $\kappa$, we say that a child $t'$ of $t$ is $\kappa$-embeddable if there are at least $\kappa$ children $t''$ of $t$ such that $T_{t'}$ is a rooted topological minor of $T_{t''}$.
\end{definition}

\begin{corollary}
    \label{lem_rotatearoundKappa}
    Let $(T,v)$ be an infinite rooted tree, $t \in V(T)$ and $\mathcal{T} = \{T_{t'} \colon t' \in N^+(t)\}$. 
    Then for any infinite cardinal $\kappa$, the number of children of $t$ which are not $\kappa$-embeddable is less than $\kappa$. 
\end{corollary}
\begin{proof}
    By Theorem \ref{t:NW} the set $\mathcal{T} = \{T_{t'} \colon t' \in N^+(t)\}$ is well-quasi-ordered by $\leq_r$ and so the claim follows by Lemma~\ref{lem_rotateKappaCollections} applied to $\mathcal{T}$, $\leq_r$, and $\kappa$.
\end{proof}

\begin{corollary}
    \label{lem_rotatearound}
    Let $(T,v)$ be an infinite rooted tree, $t \in V(T)$ a vertex of infinite degree and $(t_i\in N^{+}(t) \colon i \in \N)$ a sequence of countably many of its children. 
    Then there exists $N_t \in \mathbb{N}$ such that for all $n \geq N_t$,
    \[ \{t\} \cup \bigcup_{i > N_t} T_{t_i} \leq_r \{t\} \cup \bigcup_{i > n} T_{t_i}\]
    (considered as trees rooted at $t$) fixing the root $t$.
\end{corollary}

\begin{proof}

    Consider a labelling $c \colon T_t \to [2]$ mapping $t$ to $1$, and all remaining vertices of $T_t$ to $2$. 
    By Theorem \ref{t:Laver}, the set $\mathcal{T} = \{\{t\} \cup \bigcup_{i> n } T_{t_i} \colon n \in \N\}$ is well-quasi-ordered by $\leq_r$ respecting the labelling, and so the claim follows by applying Lemma~\ref{lem_rotateKappaCollections} to $\mathcal{T}$ and $\leq_r$ with $\kappa = \aleph_0$.
\end{proof}


%
%


\begin{definition}[Self-similarity]
    \label{def_selfsimilar}
    A ray $R = r_1r_2r_3 \ldots $ in a rooted tree $(T,v)$ which is upwards with respect to the tree order {\em displays self-similarity of $T$} if there are infinitely many $n$ such that there exists a subdivision $H$ of $T_{r_0}$ with $H \subseteq_r T_{r_n}$ and $H(R) \subseteq R$.
\end{definition}

\begin{corollary}
    \label{lem_pushalong}
    Let $(T,v)$ be an infinite rooted tree and let $R = r_1 r_2 r_3 \ldots$ be a ray which is upwards with respect to the tree order. 
    Then there is a $k \in \mathbb{N}$ such that $r_kR$ displays self-similarity of $T$.\footnote{A slightly weaker statement, without the additional condition that $H(R) \subseteq R$ appeared in \cite[Lemma~1]{A79}.}
\end{corollary}
\begin{proof}
    Consider a labelling $c \colon T \to [2]$ mapping the vertices on the ray $R$ to $1$, and labelling all remaining vertices of $T$ with $2$. 
    By Theorem \ref{t:Laver}, the set $\mathcal{T} = \{(T_{r_i},c_i) \colon i \in \mathbb{N}\}$, where $c_i$ is the natural restriction of $c$ to $T_{r_i}$, is well-quasi-ordered by $\leq_r$ respecting the labellings. 
    Hence by Lemma~\ref{lem_rotateKappaCollections}, the number of indices $i$ such that $T_{r_i}$ is not $\aleph_0$-embeddable in $\mathcal{T}$ is finite. 
    Let $k$ be larger than any such $i$. 
    Then, since $T_{r_k}$ is $\aleph_0$-embeddable in $\mathcal{T}$, there are infinitely many $r_j \in r_kR$ such that $T_{r_k} \leq_r T_{r_j}$ respecting the labelling, 
    i.e.\ mapping the ray to the ray, and hence $r_kR$ displays the self similarity of $T$.
\end{proof}

\section{Linkages between rays}
\label{sec4}
In this section we will establish a toolkit for constructing a disjoint system of paths from one  family of disjoint rays to another.

\begin{definition}[Tail of a ray]
    Given a ray $R$ in a graph $\Gamma$ and a finite set $X \subseteq V(\Gamma)$ the \emph{tail of $R$ after $X$}, denoted by $T(R,X)$, is the unique infinite component of $R$ in $\Gamma - X$.
\end{definition}

\begin{definition}[Linkage of families of rays]
    \label{d:linkage}
    Let $\mathcal{R} = (R_i \colon i\in I)$ and $\mathcal{S} = (S_j \colon j \in J)$ be families of vertex disjoint rays, where the initial vertex of each $R_i$ is denoted $x_i$. 
    A~family of paths $\mathcal{P} = (P_i \colon i \in I)$, is a \emph{linkage} from $\mathcal{R}$ to $\mathcal{S}$ if there is an injective function $\sigma \colon I \rightarrow J$ such that
    \begin{itemize}
        \item each $P_i$ joins a vertex $x'_i \in R_i$ to a vertex $y_{\sigma(i)} \in S_{\sigma(i)}$;
        \item the family $\mathcal{T} = (x_i R_i x'_i P_i y_{\sigma(i)} S_{\sigma(i)} \colon i\in I)$ is a collection of disjoint rays.
    \end{itemize}
    We say that $\Tcal$ is obtained by {\em transitioning} from $\Rcal$ to $\Scal$ along the linkage $\Pcal$.
    Given a finite set of vertices $X \subseteq V(\Gamma)$, we say that $\Pcal$ is \emph{after $X$} if $x'_i \in T(R_i,X)$ and $x'_i P_i y_{\sigma(i)} S_{\sigma(i)}$ avoids $X$ for all $i \in I$.
\end{definition}

\begin{lemma}[Weak linking lemma]
    \label{l:weaklink}
    Let $\Gamma$ be a graph and $\epsilon \in \Omega(\Gamma)$. Then for any families $\Rcal = (R_i \colon i \in [n])$ and $\Scal = (S_j \colon j \in [n])$ of vertex disjoint rays in $\epsilon$ and any finite set $X$ of vertices, there is a linkage from $\mathcal{R}$ to $\mathcal{S}$ after $X$.
\end{lemma}
\begin{proof}
    Let us write $x_i$ for the initial vertex of each $R_i$ and let $x'_i$ be the initial vertex of the tail $T(R_i,X)$. 
    Furthermore, let $X' = X \cup \bigcup_{i \in [n]} R_ix'_i$. 
    For $i \in [n]$ we will construct inductively finite disjoint connected subgraphs $K_i \subseteq \Gamma$ for each $i \in [n]$ such that
    \begin{itemize}
        \item $K_i$ meets $T(S_j,X')$ and $T(R_j,X')$ for every $j \in [n]$;
        \item $K_i$ avoids $X'$.
    \end{itemize}
    Suppose that we have constructed $K_1, \dots, K_{m-1}$ for some $m \le n$. 
    Let us write $X_m = X' \cup\bigcup_{i<m} V(K_i)$. 
    Since $R_1,\ldots, R_n$ and $S_1,\ldots, S_n$ lie in the same end $\epsilon$, there exist paths $Q_{i,j}$ between $T(R_i,X_m)$ and $T(S_j,X_m)$  avoiding $X_m$ for all $i \neq j \in [n]$. 
    Let $K_m=F \cup \bigcup_{i \neq j \in [n]} Q_{i,j}$, where $F$ consists of an initial segment of each $T(R_i,X_m)$ sufficiently large to make $K_m$ connected. 
    Then it is clear that $K_m$ is disjoint from all previous $K_i$ and satisfies the claimed properties.

    Let $K = \bigcup_{i=1}^n K_i$ and for each $j \in [n]$, let $y_j$ be the initial vertex of $T(S_j,V(K))$. 
    Note that by construction $T(S_j,V(K))$ avoids $X$ for each $j$, since $K_1$ meets $T(S_j,X)$ and so $T(S_j,V(K)) \subseteq T(S_j,X) $.

    We claim that there is no separator of size $<n$ between $\{x'_1,\ldots,x'_n\}$ and $\{y_1,\ldots,y_n\}$ in the subgraph $\Gamma' \subseteq \Gamma$ where $\Gamma' = K \cup \bigcup_{j=1}^n T(R_j, X') \cup T(S_j, X')$. 
    Indeed, any set of $<n$ vertices must avoid at least one ray $R_i$, at least one graph $K_m$ and one ray $S_j$. 
    However, since $K_m$ is connected and meets $R_i$ and $S_j$, the separator does not separate $x'_i$ from $y_j$.

    Hence, by a version of Menger's theorem for infinite graphs \cite[Proposition 8.4.1]{D16}, there is a collection of $n$ disjoint paths $P_i$ from $x'_i$ to $y_{\sigma(i)}$ in $\Gamma'$. 
    Since $\Gamma'$ is disjoint from $X$ and meets each $R_i x'_i$ in $x'_i$ only, it is clear that $\mathcal{P} = (P_i \colon i \in [n])$ is as desired.
\end{proof}

In some cases we will need to find linkages between families of rays which avoid more than just a finite subset $X$. 
For this we will use the following lemma, which is stated in slightly more generality than needed in this paper. 
Broadly the idea is that if we have a family of disjoint rays $(R_i \colon i \in [n])$ tending to an end $\epsilon$ and a number $a \in \mathbb{N}$, then there is some fixed number $N=N(a,n)$ such that if we have $N$ disjoint graphs $H_i$, each with a specified ray $S_i$ tending to $\epsilon$, then we can `re-route' the rays $(R_i \colon i \in [n])$ to some of the rays $(S_j \colon j \in [N])$, in such a way that we totally avoid $a$ of the graphs $H_i$. 

\begin{lemma}[Strong linking lemma]\label{l:link}
    Let $\Gamma$ be a graph and $\epsilon \in \Omega(\Gamma)$.
    Let $X$ be a finite set of vertices, $a,n \in \mathbb{N}$, and $\Rcal = (R_i \colon i \in [n])$ a family of vertex disjoint rays in $\epsilon$. 
    Let $x_i$ be the initial vertex of $R_i$ and let $x'_i$ the initial vertex of the tail $T(R_i,X)$. 
    
    Then there is a finite number $N = N(\Rcal,X,a)$ with the following property: For every collection $(H_j \colon j\in[N])$ of vertex disjoint subgraphs of $\Gamma$, all disjoint from $X$ and each including a specified ray $S_j$ in $\epsilon$, there is a set $A \subseteq [N]$ of size $a$ and a linkage $\Pcal = (P_i \colon i \in [n])$ from $\Rcal$ to $(S_j \colon j \in [N])$ 
    which is after $X$ and such that the family 
    \[
        \mathcal{T} = \left(x_iR_ix'_iP_iy_{\sigma(i)}S_{\sigma(i)} \colon i\in [n]\right)
    \]
    avoids $\bigcup_{k \in A} H_k$. 
\end{lemma}
\begin{proof}
Let $X' = X \cup \bigcup_{i \in [n]} R_ix'_i$ and let $N_0 = |X'|$. We claim that the lemma holds with $N = N_0 + n^3 + a$.

    Indeed suppose that $(H_j \colon j\in[N])$ is a collection of vertex disjoint subgraphs as in the statement of the lemma. 
    Since the $H_j$ are vertex disjoint, we may assume without loss of generality that the family $(H_j \colon j\in [n^3+a])$ is disjoint from $X'$.

    For each $i \in [n^2]$ we will build inductively finite, connected, vertex disjoint subgraphs $\hat{K}_i$ such that 
    \begin{itemize}
        \item $\hat{K}_i$ contains $x'_{i\pmod{n}}$, 
        \item $\hat{K}_i$ meets exactly $n$ of the $H_j$, that is $|\{ \, j \in [n^3+a] \,:\, \hat{K}_i \cap H_j \neq \emptyset\}| = n$, and
        \item $\hat{K}_i$ avoids $X'$.
    \end{itemize}

    Suppose we have done so for all $i < m$. Let $X_m = X' \cup \bigcup_{i<m} V(\hat{K}_i)$. 
    We will build inductively for $t = 0, \ldots,n$ increasing connected subgraphs $\hat{K}^t_m$ that meet $R_{i\pmod{n}}$, meet exactly $t$ of the $H_j$, and avoid $X_m$. 

    We start with $\hat{K}^0_m = \emptyset$. 
    For each $t=0,\ldots n-1$, if $T(R_{m \pmod{n}},X_m)$ meets some $H_j$ not met by $\hat{K}^t_m$ then there is some initial vertex $z_t \in T(R_{m \pmod{n}},X_m)$ where it does so and we set $\hat{K}^{t+1}_m := \hat{K}^t_m \cup T(R_{m \pmod{n}},X_m)z_t$. 
    Otherwise we may assume $T(R_{m \pmod{n}},X_m)$ does not meet any such $H_j$. In this case, let $j \in [n^3+a]$ be such that $\hat{K}^t_m \cap H_j = \emptyset$. 
    Since $R_{m \pmod{n}}$ and $S_j$ belong to the same end $\epsilon$, there is some path $P$ between $T(R_{m \pmod{n}},X_m)$ and $T(S_j,X_m)$ which avoids $X_m$. 
    Since this path meets some $H_k$ with $k \in [n^3+a]$ which $\hat{K}^t_m$ does not, there is some initial segment $P'$ which meets exactly one such $H_k$. 
    To form $\hat{K}^{t+1}_m$ we add this path to $\hat{K}^t_m$ together with an appropriately large initial segment of $T(R_{m \pmod{n}},X_m)$ such that $\hat{K}^{t+1}_m$ is connected and contains $x'_{m\pmod{n}}$. 
    Finally we let $\hat{K}_m = \hat{K}^n_m$.

    Let $K = \bigcup_{i \in [n^2]} \hat{K}_i$. Since each $\hat{K}_i$ meets exactly $n$ of the $H_j$, the set 
    \[J = \{ j\in [n^3+a] \,: \, H_j \cap K \neq \emptyset\}\] 
    satisfies $|J| \leq n^3$. 
    For each $j \in J$ let $y_j$ be the initial vertex of $T(S_j,V(K))$. 

    We claim that there is no separator of size $<n$ between $\{x'_1,\ldots x'_n \}$ and $\{ y_j \, : \,j \in J\}$ in the subgraph $\Gamma' \subseteq \Gamma$ where $\Gamma' = K \cup \bigcup_{j \in [n]} T(R_j, X') \cup \bigcup_{j \in J} H_j$. 
    Suppose for a contradiction that there is such a separator $S$. 
    Then $S$ cannot meet every $R_i$, and hence avoids some $R_q$. 
    Furthermore, there are $n$ distinct $\hat{K}_i$ such that $i = q \pmod{n}$, all of which are disjoint.
    Hence there is some $\hat{K}_r$ with $r= q \pmod{n}$ disjoint from $S$. 
    Finally, $|\{ j \in J \, : \, \hat{K}_r \cap H_j \neq \emptyset\}| = n$ and so there is some $H_s$ disjoint from $S$ such that $\hat{K}_r \cap H_s \neq \emptyset$. 
    Since $\hat{K}_r$ meets $T(R_q,X')$ and $H_s$, there is a path from $x'_q$ to $y_s$ in $\Gamma'$, contradicting our assumption.

    Hence, by a version of Menger's theorem for infinite graphs \cite[Proposition 8.4.1]{D16}, there is a family of disjoint paths $\Pcal = (P_i \colon i \in [n])$ in $\Gamma'$ from $x'_i$ to $y_{\sigma(i)}$. 
    Furthermore, since $|J| \leq n^3$ there is some subset $A \subseteq [n^3+a]$ of size $a$ such that $H_k$ is disjoint from $K$ for each $k \in A$. 

    Therefore, since $\Gamma'$ is disjoint from $X'$ and meets each $R_i x'_i$ in $x'_i$ only, the family $\Pcal$ is a linkage from $\Rcal$ to $(S_j)_{j \in [n^3+a]}$ which is after $X$ such that 
    \[
        \mathcal{T} = \left(x_iR_ix'_iP_iy_{\sigma(i)}S_{\sigma(i)} \colon i \in [n] \right)
    \]
    avoids $\bigcup_{k \in A} H_k$.
\end{proof}

We will also need the following result, which allows us to work with paths instead of rays if the end $\epsilon$ is dominated by infinitely many vertices.

\begin{lemma}\label{raystopaths}
    Let $\Gamma$ be a graph and $\epsilon$ an end of $\Gamma$ which is dominated by infinitely many vertices. 
    Let $x_1, x_2, \ldots, x_k$ be distinct vertices. 
    If there are disjoint rays from the $x_i$ to $\epsilon$ then there are disjoint paths from the $x_i$ to distinct vertices $y_i$ which dominate $\epsilon$.
\end{lemma}
\begin{proof}
    We argue by induction on $k$. The base case $k = 0$ is trivial, so let us assume $k > 0$.

    Consider any family of disjoint rays $R_i$, each from $x_i$ to $\epsilon$. 
    Let $y_k$ be any vertex dominating $\epsilon$. 
    Let $P$ be a $y_k$\,--\,$\bigcup_{i = 1}^kR_i$\,-path. 
    Without loss of generality the endvertex $u$ of $P$ in $\bigcup_{i = 1}^kR_i$ lies on $R_k$. 
    Then by the induction hypothesis applied to the graph $\Gamma - R_kuP$ we can find disjoint paths in that graph from the $x_i$ with $i < k$ to vertices $y_i$ which dominate $\epsilon$. 
    These paths together with $R_kuP$ then form the desired collection of paths.
\end{proof}

To go back from paths to rays we will use the following lemma. 

\begin{lemma}\label{pathstorays}
    Let $\Gamma$ be a graph and $\epsilon$ an end of $\Gamma$ which is dominated by infinitely many vertices. 
    Let $y_1, y_2, \ldots, y_k$ be vertices, not necessarily distinct, dominating $\Gamma$. 
    Then there are rays $R_i$ from the respective $y_i$ to $\epsilon$ which are disjoint except at their initial vertices.
\end{lemma}
\begin{proof}
    We recursively build for each $n\in\Nbb$ paths $P_1^n, \ldots, P_k^n$, each $P_i^n$ from 
    $y_i$ to a vertex $y_i^n$ dominating $\epsilon$, disjoint except at their initial vertices, such that for $m < n$ each $P_i^n$ properly extends $P_i^m$. 
    We take $P_i^0$ to be a trivial path.
    For $n > 0$, build the $P_i^n$ recursively in $i$: 
    To construct $P_i^n$, we start by taking $X_i^n$ to be the finite set of all the vertices of the $P_j^n$ with $j < i$ or $P_j^{n-1}$ with $j \geq i$. 
    We then choose a vertex $y_i^n$ outside of $X_i^n$ which dominates $\epsilon$ and a path $Q_i^n$ from $y_i^{n-1}$ to $y_i^n$ internally disjoint from $X_i^n$. 
    Finally we let $P_i^n := P_i^{n-1}y_{n-1}Q_i^n$. 

    Finally, for each $i \leq k$, we let $R_i$ be the ray $\bigcup_{n \in \Nbb} P_i^n$. 
    Then the $R_i$ are disjoint except at their initial vertices, and they are in $\epsilon$, since each of them contains infinitely many dominating vertices of $\epsilon$.
\end{proof}

\section{%
\texorpdfstring{$G$-tribes and concentration of $G$-tribes towards an end}%
{G-tribes and concentration of G-tribes towards an end}%
}
\label{secGtribes}

For showing that a given graph $G$ is ubiquitous with respect to a fixed relation $\triangleleft$, we shall assume that $n G\triangleleft \Gamma$ for every $n \in \N$ and need to show that this implies that $\aleph_0 G\triangleleft \Gamma$. 
Since each subgraph witnessing that $n G \triangleleft \Gamma$ will be a collection of $n$ disjoint subgraphs each being a witness for $G \triangleleft \Gamma$, it will be useful to introduce some notation for talking about these families of collections of $n$ disjoint witnesses for each $n$.

To do this formally, we need to distinguish between a relation like the topological minor relation and the subdivision relation. 
Recall that we write $G \leq^* H$ if $H$ is a subdivision of $G$ and $G \leq \Gamma$ if $G$ is a topological minor of $\Gamma$.
We can interpret the topological minor relation as the composition of the subdivision relation and the subgraph relation.

Given two relations $R$ and $S$, let their \emph{composition $S \circ R$} be the relation defined by $x(S \circ R)z$ if and only if there is a $y$ such that $xRy$ and $ySz$. 

Hence we have that
$G \,(\subseteq \circ \leq^{*})\, \Gamma$ if and only if there exists $H$ such that $G \leq^* H \subseteq \Gamma$, that is, if and only if $G \leq \Gamma$.

While in this paper we will only work with the topological minor relation, we will state the following definition and lemmas in greater generality, so that we may apply them in later papers in this series \cite{BEEGHPTII,BEEGHPTIII,BEEGHPTIV}.

In general, we want to consider a pair $(\triangleleft$, $\blacktriangleleft)$ of binary relations of graphs with the following properties.
\begin{enumerate}[label={(R\arabic*)}]
    \item $\triangleleft\ =\ (\subseteq \circ \blacktriangleleft)$;
    \item \label{d:relationcomp} Given a set $I$ and a family $(H_i : i \in I)$ of pairwise disjoint graphs with $G \blacktriangleleft H_i$ for all $i \in I$, then ${\lvert I \rvert \cdot G \blacktriangleleft \bigcup \{H_i : i \in I \}}$.
\end{enumerate}
We call a pair $(\triangleleft, \blacktriangleleft)$ with these properties \emph{compatible}.

Other examples of compatible pairs are $(\subseteq, \cong)$, where $\cong$ denotes the isomorphism relation, as well as $(\preceq, \preceq^*)$, where $G \preceq^* H$ if $H$ is an inflated copy of $G$.

\begin{definition}[\Gtribe s]
    Let $G$ and $\Gamma$ be graphs, and let $(\triangleleft, \blacktriangleleft)$ be a compatible pair of relations between graphs.
    \begin{itemize}
        \item A \emph{\Gtribe\ in $\Gamma$ (with respect to $(\triangleleft, \blacktriangleleft)$)} is a collection $\mathcal{F}$ of finite sets $F$ of disjoint subgraphs~$H$ of $\Gamma$ such that $G \blacktriangleleft H$ for each \emph{member of $\Fcal$} $H \in \bigcup \Fcal$.
        \item A \Gtribe\ $\mathcal{F}$ in $\Gamma$ is called \emph{thick}, if for each $n \in \mathbb{N}$ there is a \emph{layer} $F \in \mathcal{F}$ with $|F| \geq n$; 
            otherwise, it is called \emph{thin}.\footnote{A similar notion of \emph{thick} and \emph{thin families} was also introduced by Andreae in \cite{A79} (in German) and in \cite{A13}. 
            The remaining notions, and in particular the concept of a \emph{concentrated $G$-tribe}, which will be the backbone of essentially all our results in this series of papers, is new.} 

        \item A \Gtribe\ $\mathcal{F}'$ in $\Gamma$ is a \emph{\Gsubtribe} of a \Gtribe\ $\mathcal{F}$ in $\Gamma$, denoted by $\Fcal' \triangleleft \Fcal$, 
            if there is an injection $\Psi\colon\Fcal'\to \Fcal$ such that for each $F' \in \mathcal{F}'$ there is an injection $\varphi_{F'} \colon F' \to \Psi(F')$ such that $V(H') \subseteq V(\varphi_{F'}(H'))$ for each $H' \in F'$. 
            The \Gsubtribe \ $\Fcal'$ is called \emph{flat}, denoted by $\Fcal' \subseteq \Fcal$, if there is such an injection $\Psi$ satisfying $F' \subseteq \Psi(F')$.

        \item A thick $G$-tribe $\mathcal{F}$ in $\Gamma$ is \emph{concentrated at an end $\epsilon$} of $\Gamma$, if  for every finite vertex set $X$ of $\Gamma$, the $G$-tribe $\Fcal_X = \{F_X \colon F \in \Fcal \}$ consisting of the layers $F_X=\{H \in F: H \not\subseteq C(X,\epsilon)\}\subseteq F$ is a thin subtribe of $\Fcal$.
    \end{itemize}
\end{definition}
Hence, for a given compatible pair $(\triangleleft, \blacktriangleleft)$, if we wish to show that $G$ is $\triangleleft$-ubiquitous, we will need to show that the existence of a thick $G$-tribe in $\Gamma$ with respect to $(\triangleleft, \blacktriangleleft)$ implies $\aleph_0 G\triangleleft \Gamma$. 
We first observe that removing a thin $G$-tribe from a thick $G$-tribe always leaves a thick $G$-tribe.
\begin{lemma}[{cf.~\cite[Lemma~3]{A79} or \cite[Lemma~2]{A13}}]
    \label{l:removethin}
    Let $\Fcal$ be a thick $G$-tribe in $\Gamma$ and let $\Fcal'$ be a thin subtribe of $\Fcal$, witnessed by $\Psi\colon \Fcal'\to \Fcal$ and $(\varphi_{F'} \colon F' \in \mathcal{F}')$. 
    For $F\in \Fcal$, if $F\in \Psi(\Fcal')$, let $\Psi^{-1}(F)=\{F'_F\}$ and set $\hat{F}=\varphi_{F'_F}(F'_F)$. If $F\notin \Psi(\Fcal')$, set $\hat{F}=\emptyset$. Then \[\Fcal'':=\{F\setminus \hat{F}\colon F\in \Fcal\}\] is a thick flat $G$-subtribe of $\Fcal$.
\end{lemma}
\begin{proof}
    $\Fcal''$ is obviously a flat subtribe of $\Fcal$. 
    As $\Fcal'$ is thin, there is a $k\in \Nbb$ such that $|F'|\le k$ for every $F'\in \Fcal'$. Thus $|\hat{F}|\le k$ for all $F\in \Fcal$.
    Let $n\in \Nbb$. 
    As $\Fcal$ is thick, there is a layer $F\in \Fcal$ satisfying $|F|\ge n+k$. 
    Thus $|F\setminus \hat{F}|\ge n+k-k=n$.
\end{proof}
Given a thick $G$-tribe, the members of this tribe may have different properties, for example, some of them contain a ray belonging to a specific end $\epsilon$ of $\Gamma$ whereas some of them do not. 
The next lemma allows us to restrict onto a thick subtribe, in which all members have the same properties, as long as we consider only finitely many properties. 
E.g. we find a subtribe in which either all members contain an $\epsilon$-ray, or none of them contain such a ray. 

\begin{lemma}[Pigeon hole principle for thick $G$-tribes]
    \label{Lem_finitechoice}
    Suppose for some $k \in \N$, we have a $k$-colouring $c\colon \bigcup \mathcal{F} \to [k]$ of the members of some thick $G$-tribe $\mathcal{F}$ in $\Gamma$. 
    Then there is a monochromatic, thick, flat $G$-subtribe $\mathcal{F}'$ of $\mathcal F$.
\end{lemma}

\begin{proof}
    Since $\mathcal{F}$ is a thick $G$-tribe, there is a sequence $(n_i \colon i \in \N)$ of natural numbers and a sequence $(F_i \in \Fcal \colon i \in \N)$ such that 
    \[ n_1 \leq |F_1| < n_2 \leq |F_2| < n_3 \leq |F_3| < \cdots. \]
    Now for each $i$, by pigeon hole principle, there is one colour $c_i \in [k]$ such that the subset $F' _i \subseteq F_i$ of elements of colour $c_i$ has size at least $n_i / k$. 
    Moreover, since $[k]$ is finite, there is one colour $c^* \in [k]$ and an infinite subset $I \subseteq \N$ such that $c_i=c^*$ for all $i \in I$. 
    But this means that $\mathcal{F}' := \{F'_i \colon i \in I \}$ is a monochromatic, thick, flat $G$-subtribe.
\end{proof}

In this series of papers we will be interested in graph relations such as $\subseteq$, $\leq$ and $\preceq$. 
Given a connected graph $G$ and a compatible pair of relations $(\triangleleft,\blacktriangleleft)$ we say that a $G$-tribe $\Fcal$ w.r.t $(\triangleleft,\blacktriangleleft)$ is \emph{connected} if every member $H$ of $\Fcal$ is connected. Note that for relations $\blacktriangleleft$ like $\cong, \le^ *, \preceq^*$, if $G$ is connected and $G \blacktriangleleft H$, then $H$ is connected. 
In this case, any $G$-tribe will be connected.

\begin{lemma}\label{l:concentrated}
    Let $G$ be a connected graph (of arbitrary cardinality), $(\triangleleft,\blacktriangleleft)$ a compatible pair of relations of graphs and $\Gamma$ a graph containing a thick connected $G$-tribe $\mathcal{F}$ w.r.t.\ $(\triangleleft,\blacktriangleleft)$. 
    Then either $\aleph_0 G \triangleleft \Gamma$, or there is a thick flat subtribe $\Fcal'$ of $\Fcal$ and an end $\epsilon$ of $\Gamma$ such that $\Fcal'$ is concentrated at $\epsilon$.
\end{lemma}

\begin{proof}
    For every finite vertex set $X \subseteq V(\Gamma)$, only a thin subtribe of $\Fcal$ can meet $X$, so by Lemma \ref{l:removethin} a thick flat subtribe $\Fcal''$ is contained in the graph $\Gamma-X$. 
    Since each member of $\Fcal''$ is connected, any member $H$ of $\Fcal''$ is contained in a unique component of $\Gamma -X$. 
    If for any $X$, infinitely many components of $\Gamma - X$ contain a $\blacktriangleleft$-copy of $G$, the union of all these copies is a $\blacktriangleleft$-copy of $\aleph_0G$ in $\Gamma$ by \ref{d:relationcomp}, hence $\aleph_0G\triangleleft \Gamma$. 
    Thus, we may assume that for each $X$, only finitely many components contain elements from $\Fcal''$, and hence, by colouring each $H$ with a colour corresponding to the component of $\Gamma -X$ containing it, we may assume by the pigeon hole principle for $G$-tribes, Lemma~\ref{Lem_finitechoice}, that at least one component of $\Gamma-X$ contains a thick flat subtribe of $\Fcal$. 
    
    Let $C_0 = \Gamma$ and $\Fcal_0 = \Fcal$ and consider the following recursive process: 
    If possible, we choose a finite vertex set $X_{n}$ in $C_n$ such that there are two components $C_{n+1} \neq D_{n+1}$ of $C_n - X_n$ where $C_{n+1}$ contains a thick flat subtribe $\Fcal_{n+1} \subseteq \Fcal_n$ and $D_{n+1}$ contains at least one $\blacktriangleleft$-copy $H_{n+1}$ of $G$. 
    Since by construction all $H_n$ are pairwise disjoint, we either find infinitely many such $H_n$ and thus, again by \ref{d:relationcomp}, an $\aleph_0 G \triangleleft \Gamma$, or our process terminates at step $N$ say. 
    That is, we have a thick flat subtribe $\Fcal_{N}$ contained in a subgraph $C_N$ such that there is no finite vertex set $X_N$ satisfying the above conditions.

    Let $\Fcal' := \Fcal_{N}$. We claim that for every finite vertex set $X$ of $\Gamma$, there is a unique component $C_{X}$ of $\Gamma - X$ that contains a thick flat $G$-subtribe of $\Fcal'$. 
    Indeed, note that if for some finite $X \subseteq \Gamma$ there are two components $C$ and $C'$ of $\Gamma - X$ both containing thick flat $G$-subtribes of $\Fcal'$, then since every $G$-copy in $\Fcal'$ is contained in $C_N$, it must be the case that $C \cap C_N \neq \emptyset \neq C' \cap C_N$. 
    But then $X_{N} =X \cap C_N \neq \emptyset$ is a witness that our process could not have terminated at step $N$. 
 
    Next, observe that whenever $X' \supseteq X $, then $C_{X'} \subseteq C_X$. By a theorem of Diestel and K\"uhn, \cite{Ends}, it follows that there is a unique end $\epsilon$ of $\Gamma$ such that $C(X,\epsilon) = C_X$ for all finite $X \subseteq \Gamma$. 
    It now follows easily from the uniqueness of $C_X = C(X,\epsilon)$ that $\Fcal'$ is concentrated at this $\epsilon$. 
\end{proof} 
We note that concentration towards an end $\epsilon$ is a robust property in the following sense:

\begin{lemma}
    \label{lem_subtribesinheritconcentration}
    Let $G$ be a connected graph (of arbitrary cardinality), $(\triangleleft,\blacktriangleleft)$ a compatible pair of relations of graphs and $\Gamma$ a graph containing a thick connected $G$-tribe $\mathcal{F}$ w.r.t.\ $(\triangleleft,\blacktriangleleft)$ concentrated at an end $\epsilon$ of $\Gamma$. 
    Then the following assertions hold: 
    \begin{enumerate}
        \item For every finite set $X$, the component $C(X,\epsilon)$ contains a thick flat $G$-subtribe of~$\Fcal$.
        \item Every thick subtribe $\Fcal'$ of $\Fcal$ is concentrated at $\epsilon$, too.
    \end{enumerate}
\end{lemma} 

\begin{proof}
    Let $X$ be a finite vertex set. By definition, if the $G$-tribe $\Fcal$ is concentrated at $\epsilon$, then $\Fcal$ is thick, and the subtribe $\Fcal_X$ consisting of the sets $F_X=\{H \in F: H \not\subseteq C(X,\epsilon)\}\subseteq F$ for $F \in \Fcal$ is a thin subtribe of $\Fcal$, 
    i.e.\ there exists $k \in \N$ such that $|F_X| \leq k$ for all $F_X \in \Fcal_X$.

    For $(1)$, observe that the $G$-tribe $\Fcal' = \{F\setminus F_X\colon F \in \Fcal\}$ is a thick flat subtribe of $\Fcal$ by Lemma \ref{l:removethin}, and all its members are contained in $C(X,\epsilon)$ by construction. 

    For $(2)$, observe that if $\Fcal'$ is a subtribe of $\Fcal$, then for every $F' \in \Fcal'$ there is an injection $\varphi_{F'} \colon F' \to F$ for some $F \in \Fcal$. 
    Therefore, $|\varphi_{F'}^{-1}(F_X)|\leq k$ for $F_X \subseteq F$ as defined above, and so only a thin subtribe of $\Fcal'$ is not contained in $C(X,\epsilon)$.
\end{proof}

\section{Countable subtrees}
\label{sec:countable-subtrees}

In this section we prove Theorem~\ref{t:countembed}. 
Let $S$ be a countable subtree of $T$. 
Our aim is to construct an {\em $S$-horde} $(Q_i \colon i \in \N)$ of disjoint suitable $S$-minors in $\Gamma$ inductively. 
By Lemma~\ref{l:concentrated}, we may assume without loss of generality that there are an end $\epsilon$ of $\Gamma$ and a thick $T$-tribe $\mathcal{F}$ concentrated at $\epsilon$. 

In order to ensure that we can continue the construction at each stage, we will require the existence of additional structure for each $n$. 
But the details of what additional structure we use will vary depending on how many vertices dominate $\epsilon$. 
So, after a common step of preprocessing, in Section~\ref{sec7.1}, the proof of Theorem~\ref{t:countembed} splits into two cases according to whether the number of $\epsilon$-dominating vertices in $\Gamma$ is finite (Section~\ref{sec7.2}) or infinite (Section~\ref{sec7.3}). 

\subsection{Preprocessing}
\label{sec7.1}

We begin by picking a root $v$ for $S$, and also consider $T$ as a rooted tree with root $v$. Let $V_{\infty}(S)$ be the set of vertices of infinite degree in $S$.

\begin{definition}
    Given $S$ and $T$ as above, define a spanning locally finite forest $S^*\subseteq S$ by
    \[S^* := S \setminus \bigcup_{t \in V_\infty(S)} \{tt_i \colon t_i \in N^+(t), i > N_t \},\] 
    where $N_t$ is as in Corollary~\ref{lem_rotatearound}. 
    We will also consider every component of $S^*$ as a rooted tree given by the induced tree order from $T$.
\end{definition}

\begin{definition}
    \label{def_extensionedge}
    An edge $e$ of $S^*$ is an {\em extension edge} if there is a ray in $S^*$ starting at $e^+$ which displays self-similarity of $T$. 
    For each extension edge $e$ we fix one such a ray $R_e$. Write $Ext(S^*) \subseteq E(S^*)$ for the set of extension edges.
\end{definition}

Consider the forest $S^* - Ext(S^*)$ obtained from $S^*$ by removing all extension edges. 
Since every ray in $S^*$ must contain an extension edge by Corollary~\ref{lem_pushalong}, each component of $S^* - Ext(S^*)$ is a locally finite rayless tree and so is finite (this argument is inspired by \cite[Lemma~2]{A79}). 
We enumerate the components of $S^* - Ext(S^*)$ as $S^*_0,S^*_1,\ldots$ in such a way that for every $n \geq 0$, the set 
\[
    S_{n} := S\left[\bigcup_{i \leq n } V(S^*_i) \right]
\]
is a finite subtree of $S$ containing the root $r$. 
Let us write $\partial (S_n) =E_{S^*}(S_n,S^* \setminus S_n)$, and note that $\partial (S_n) \subseteq Ext(S^*)$. We make the following definitions:

\begin{itemize}
    \item For a given $T$-tribe $\Fcal$ and ray $R$ of $T$, we say that $R$ {\em converges to $\epsilon$ according to $\Fcal$} if for all members $H$ of $\Fcal$ the ray $H(R)$ is in $\epsilon$. 
        We say that $R$ is {\em cut from $\epsilon$ according to $\Fcal$} if for all members $H$ of $\Fcal$ the ray $H(R)$ is not in $\epsilon$. 
        Finally we say that $\Fcal$ {\em determines whether $R$ converges to $\epsilon$} if either $R$ converges to $\epsilon$ according to $\Fcal$ or $R$ is cut from $\epsilon$ according to $\Fcal$. 
      \item Similarly, for a given $T$-tribe $\Fcal$ and vertex $t$ of $T$, we say that $t$ {\em dominates $\epsilon$ according to $\Fcal$} if for all members $H$ of $\Fcal$ the vertex $H(t)$ dominates $\epsilon$. 
          We say that $t$ is {\em cut from $\epsilon$ according to $\Fcal$} if for all members $H$ of $\Fcal$ the vertex $H(t)$ does not dominate $\epsilon$. 
          Finally we say that $\Fcal$ {\em determines whether $t$ dominates $\epsilon$} if either $t$ dominates $\epsilon$ according to $\Fcal$ or $t$ is cut from $\epsilon$ according to $\Fcal$.
    \item Given $n \in \N$, we say a thick $T$-tribe $\Fcal$ \emph{agrees about $\partial (S_n)$} if for each extension edge $e \in \partial (S_n)$, it determines whether $R_e$ converges to $\epsilon$. 
        We say that it \emph{agrees about $V(S_n)$} if for each vertex $t$ of $S_n$, it determines whether $t$ dominates $\epsilon$. 
    \item Since $\partial (S_n)$ and $V(S_n)$ are finite for all $n$, it follows from Lemma~\ref{Lem_finitechoice} that  given some $n \in \N$, any thick $T$-tribe has a flat thick $T$-subtribe $\Fcal$ such that $\Fcal$ agrees about $\partial (S_n)$ and $V(S_n)$. 
        Under these circumstances we set
        \begin{align*}
            \partial_\epsilon (S_n) &: = \{e \in \partial (S_n) \colon R_e \text{ converges to $\epsilon$ according to } \Fcal \}\,, \\
            \partial_{\neg \epsilon} (S_n)  &: =  \{e \in \partial (S_n) \colon R_e \text{ is cut from  $\epsilon$ according to } \Fcal \} \,, \\
             V_{\epsilon}(S_n) &:= \{t \in V(S_n) \colon t \text{ dominates $\epsilon$ according to } \Fcal \}\,, \text{ and} \\
             V_{\neg\epsilon}(S_n) &:= \{t \in V(S_n) \colon t \text{ is cut from $\epsilon$ according to } \Fcal \}\, .
        \end{align*}
     \item Also, under these circumstances, let us write $S^{\neg \epsilon}_n$ for the component of the forest $S - \partial_\epsilon (S_n) - \{e \in E_S(S_n,S \setminus S_n) \colon e^- \in V_{\epsilon}(S_n)\}$ containing the root of $S$. 
         Note that $S_n \subseteq S^{\neg \epsilon}_n$.
\end{itemize}

The following lemma contains a large part of the work needed for our inductive construction.

\begin{lemma}[$T$-tribe refinement lemma]
    \label{lem_increasingFnegEpsilon}
    Suppose we have a thick $T$-tribe $\Fcal_n$ concentrated at $\epsilon$ which agrees about $\partial (S_n)$ and $V(S_n)$ for some $n \in \N$. Let $f$ denote the unique edge from $S_n$ to $S_{n+1} \setminus S_n$. 
    Then there is a thick $T$-tribe $\Fcal_{n+1}$ concentrated at $\epsilon$ with the following properties:
    \begin{enumerate}[label=(\roman*)]
        \item \label{itemagree} $\Fcal_{n+1}$ agrees about $\partial (S_{n+1})$ and $V(S_{n+1})$.
        \item \label{itemconsistentagree} $\Fcal_{n+1} \cup \Fcal_n$ agree about $\partial (S_n) \setminus \{f\}$ and $V(S_n)$.
        \item \label{itemnested} $S^{\neg \epsilon}_{n+1} \supseteq S^{\neg \epsilon}_n$.
        \item \label{itemSeparation} For all $H \in \Fcal_{n+1}$ there is a finite $X\subseteq \Gamma$ such that $H(S^{\neg \epsilon}_{n+1}) \cap \left(X \cup C_\Gamma(X,\epsilon) \right) = H(V_\epsilon(S_{n+1}))$.
    \end{enumerate}
    Moreover, if $f \in \partial_{\epsilon} (S_n)$, and $R_f= v_0v_1v_2 \ldots \subseteq S^*$ (with $v_0=f^+$) denotes the ray displaying self-similarity of $T$ at $f$, then we may additionally assume: 
    \begin{enumerate}[label=(\roman*),resume]
        \item \label{consistentpushingalong} For every $H \in \Fcal_{n+1}$ and every $k \in \N$, there is $H' \in \Fcal_{n+1}$ with
        \begin{itemize}
            \item $H' \subseteq_r H$ 
            \item $H' (S_n) = H(S_n)$, 
            \item \label{b}$H'(T_{v_0}) \subseteq_r H(T_{v_k})$, and
            \item $H'(R_f) \subseteq H(R_f)$.
        \end{itemize}
    \end{enumerate}
\end{lemma}

\begin{proof}
    Concerning \ref{consistentpushingalong}, if $f \in \partial_{\epsilon} (S_n)$ recall that according to Definition~\ref{def_extensionedge}, the ray $R_f$ satisfies that for all $k \in \N$ we have $T_{v_0} \leq_r T_{v_k}$ such that $R_f$ gets embedded into itself. 
    In particular, there is a subtree $\hat{T}_1$ of $T_{v_1}$ which is a rooted subdivision of $T_{v_0}$ with $\hat{T}_1(R_f) \subseteq R_f$, considering $\hat{T}_1$ as a rooted tree given by the tree order in $T_{v_1}$. If we define recursively for each $k \in \N$ $\hat{T}_k = \hat{T}_{k-1}(\hat{T}_1)$ then it is clear that $(\hat{T}_k \colon k \in \N)$ is a family of rooted subdivisions of $T_{v_0}$ such that for each $k\in \N$
    \begin{itemize}
    \item $\hat{T}_k \subseteq T_{v_k}$;
    \item $\hat{T}_{k} \supseteq \hat{T}_{k+1}$;
    \item $\hat{T}_k(R_f) \subseteq R_f$
    \end{itemize}
    
    Hence, for every subdivision $H$ of $T$ with $H \in \bigcup \Fcal_n$ and every $k \in \N$, the subgraph $H(\hat{T}_k)$ is also a rooted subdivision of $T_{v_0}$. 
    Let us construct a subdivision $H^{(k)}$ of~$T$ by letting $H^{(k)}$ be the minimal subtree of $H$ containing $H(T \setminus T_{v_0}) \cup H(\hat{T}_k)$, where $H^{(k)}(T \setminus T_{v_0})=H(T \setminus T_{v_0})$ and ${H^{(k)}(T_{v_0})=H(\hat{T}_k)}$.
    Note that 
    \[
    H^{(k)}(T_{v_0}) = H(\hat{T}_k) \subseteq_r H^{(k-1)}(T_{v_0}) = H(\hat{T}_{k-1}) \subseteq_r  \ldots \subseteq_r  H(T_{v_k}). 
    \]
    
    In particular, for every subdivision $H \in \bigcup \Fcal_n$ of $T$ and every $k \in \N$, there is a subdivision $H^{(k)} \subseteq H$ of $T$ such that $H^{(k)} (S^{\neg \epsilon}_n) = H(S^{\neg \epsilon}_n)$, $H^{(k)} (T_{v_0}) \subseteq_r H(T_{v_k})$, and $H^{(k)}(R_f) \subseteq H(R_f)$.
    By the pigeon hole principle, there is an infinite index set $K_H = \{k^H_1,k^H_2,\ldots \} \subseteq \N$ such that $\{\{H^{(k)}\}\colon k \in K_H \}$ agrees about $\partial (S_{n+1})$. Consider the thick subtribe $\Fcal'_n = \{F'_i \colon F \in \Fcal_n, i \in \N\}$ of $\Fcal_n$ with 
    \[
        (\dagger) \quad \quad F'_i := \{ H^{\left(k^H_i\right)} \colon H \in F \}.
    \] %
    Observe that $\Fcal'_n \cup \Fcal_n$ still agrees about $\partial (S_n)$ and $V(S_n)$. (If $f \in \partial_{\neg \epsilon} (S_n)$, then skip this part and simply let $\Fcal'_n:=\Fcal_n$.)

    Concerning \ref{itemnested}, observe that for every $H \in \bigcup \Fcal'_n$, since the rays $H(R_e)$ for $e \in \partial_{\neg \epsilon} (S_n)$ do not tend to $\epsilon$, there is a finite vertex set $X_H$ such that $H(R_e) \cap C(X_H,\epsilon) = \emptyset$ for all $e \in \partial_{\neg \epsilon} (S_n)$. 
    Furthermore, since $X_H$ is finite, for each such extension edge $e$ there exists $x_e \in R_e$ such that 
    \[H(T_{x_e}) \cap C(X_H,\epsilon) = \emptyset.\]
    By definition of extension edges, cf.\ Definition~\ref{def_extensionedge}, for each $e \in \partial_{\neg \epsilon} (S_n)$ there is a rooted embedding of $T_{e^+}$ into $H(T_{x_e})$. 
    Hence, there is a subdivision $\tilde{H}$ of $T$ with $\tilde{H} \leq H$ and $\tilde{H}(S_n) = H(S_n)$ such that $\tilde{H}(T_{e^+}) \subseteq H(T_{x_e})$ for each $e \in \partial_{\neg \epsilon} (S_n)$.

    Note that if $e \in \partial_{\neg \epsilon} (S_n)$ and $g$ is an extension edge with $e \leq g \in \partial (S_{n+1}) \setminus \partial (S_n)$, then $\tilde{H}(R_g) \subseteq \tilde{H}(S_{e^+} ) \subseteq H(S_{x_e})$, and so
    \[
        (\ddagger) \quad \quad \tilde{H}(R_g) \text{ doesn't tend to } \epsilon. 
    \]
    Define $\tilde{\Fcal}_n$ to be the thick $T$-subtribe of $\Fcal'_n$ consisting of the $\tilde{H}$ for every $H$ in $\bigcup \Fcal'_n$.
    
    Now use Lemma~\ref{Lem_finitechoice} to chose a maximal thick flat subtribe $\Fcal^*_{n}$ of $\tilde{\Fcal}_n$ which agrees about $\partial (S_{n+1})$ and $V(S_{n+1})$, so it satisfies \ref{itemagree} and \ref{itemconsistentagree}. 
    By $(\ddagger)$, the tribe $\Fcal^*_{n}$ satisfies \ref{itemnested}, and by maximality and $(\dagger)$, it satisfies \ref{consistentpushingalong}.

    In our last step, we now arrange for \ref{itemSeparation} while preserving all other properties. 
    For each $H \in \bigcup \Fcal^*_{n}$. 
    Since $H(S_{n+1})$ is finite, we may find a finite separator $Y_H$ such that
    \[
        H(S_{n+1}) \cap \left( Y_H \cup C(Y_H,\epsilon)\right) = H(V_\epsilon(S_{n+1})).
    \]
    Since $Y_H$ is finite, for every vertex $t \in V_{\neg \epsilon}(S_{n+1})$, say with $N^+(t) = (t_i)_{i\in \mathbb{N}}$, there exists $n_t \in \N$ such that $C(Y_H,\epsilon) \cap H(T_{t_j})  = \emptyset$ for all $j \geq n_t$. 
Using Corollary~\ref{lem_rotatearound}, for every such $t$ there is a rooted embedding
    \[ 
        \{t\} \cup \bigcup_{j > N_t} T_{t_j} \leq_r \{t\} \cup \bigcup_{j > n_t} T_{t_j}.
    \]
    fixing the root $t$. 
    Hence there is a subdivision $H$' of $T$ with $H' \leq H$ such that $H'(T \setminus S) = H(T \setminus S)$ and for every $t \in V_{\neg \epsilon}(S_{n+1})$
    \[
        H'\left[\{t\} \cup \bigcup_{j > N_t} T_{t_j} \right] \cap C(Y_H,\epsilon) = \emptyset.
    \]
    Moreover, note that by construction of $\tilde{F}_n$, every such $H'$ automatically satisfies that
    \[
        H(S_{e^+}) \cap C(X_H \cup Y_H,\epsilon) = \emptyset
    \]
    for all $e \in \partial_{\neg \epsilon} (S_{n+1})$. Let $\Fcal_{n+1}$ consist of the set of $H'$ as defined above for all $H \in \Fcal^*_n$. 
    Then $X_H \cup Y_H$ is a finite separator witnessing that  $\Fcal_{n+1}$ satisfies \ref{itemSeparation}.
\end{proof}

\subsection{Only finitely many vertices dominate \texorpdfstring{$\epsilon$}{\textepsilon}}
\label{sec7.2}

We first note as in Lemma~\ref{l:concentrated}, that for every finite vertex set $X \subseteq V(\Gamma)$ only a thin subtribe of $\Fcal$ can meet $X$, so a thick subtribe is contained in the graph $\Gamma-X$. 
By removing the set of vertices dominating $\epsilon$, we may therefore assume without loss of generality that no vertex of $\Gamma$ dominates $\epsilon$.

\begin{definition}[Bounder, extender]
    Suppose that some thick $T$-tribe $\Fcal$ which is concentrated at $\epsilon$ agrees about $S_n$ for some given $n \in \N$, and $Q_1^n, Q_2^n, \ldots ,Q_n^n$ are disjoint subdivisions of $S^{\neg \epsilon}_n$ (note, $S^{\neg \epsilon}_n$ depends on $\Fcal$).
    \begin{itemize}
        \item A {\em bounder} for the $(Q^n_i \colon i \in [n])$ is a finite set $X$ of vertices in $\Gamma$ separating all the $Q_i$ from $\epsilon$, i.e.\ such that 
            \[C(X,\epsilon) \cap \bigcup_{i=1}^n Q^n_i  = \emptyset. \]
        \item An {\em extender} for the $(Q^n_i \colon i \in [n])$ is a family $\mathcal{E}_n = ( E^n_{e,i} \colon e \in \partial_\epsilon (S_n), i \in [n])$ of rays in $\Gamma$ tending to $\epsilon$ which are disjoint from each other and also from each $Q^n_i$ except at their initial vertices, and where the start vertex of $E^n_{e,i}$ is $Q^n_i(e^-)$.
    \end{itemize}
\end{definition}

To prove Theorem~\ref{t:countembed}, we now assume inductively that for some $n \in \mathbb{N}$, with $r:= \lfloor n/2 \rfloor$ and $s:= \lfloor (n+1)/2 \rfloor$ we have:
\begin{enumerate}
    \item A thick $T$-tribe $\mathcal{F}_{r}$ in $\Gamma$ concentrated at $\epsilon$ which agrees about $\partial \left(S_{r}\right)$, with a boundary $\partial_\epsilon \left(S_{r}\right)$ such that $S^{\neg \epsilon}_{r-1} \subseteq S^{\neg \epsilon}_{r}$.\footnote{Note that since $\epsilon$ is undominated, every thick $T$-tribe agrees about the fact that $V_\epsilon(S_i) = \emptyset$ for all $i \in \N$.}
    \item a family $( Q_i^n \colon i \in [s] )$ of $s$ pairwise disjoint $T$-suitable subdivisions of $S^{\neg \epsilon}_{r}$ in $\Gamma$ with $Q^n_i(S^{\neg \epsilon}_{r-1}) = Q^{n-1}_i$ for all $i \leq s-1$,
    \item a bounder $X_n$ for the $(Q^n_i \colon i \in [s])$, and 
    \item an extender $\mathcal{E}_n = ( E^n_{e,i} \colon e \in \partial_\epsilon \left(S^{\neg \epsilon}_{r}\right), \, i \in [s] )$ for the $(Q^n_i \colon i \in [s])$.
\end{enumerate}


The base case $n = 0$ it easy, as we simply may choose $\Fcal_0 \leq_r \Fcal$ to be any thick $T$-subtribe in $\Gamma$  which agrees about $\partial (S_0)$, and let all other objects be empty.

So, let us assume that our construction has proceeded to step $n \geq 0$. Our next task splits into two parts: 
First, if $n=2k-1$ is odd, we extend the already existing $k$ subdivisions $(Q^n_i \colon i \in [k])$ of $S^{\neg \epsilon}_{k-1}$ to subdivisions $(Q^{n+1}_i \colon i \in [k])$ of $S^{\neg \epsilon}_{k}$. 
And secondly, if $n=2k$ is even, we construct a further disjoint copy $Q^{n+1}_{k+1}$ of $S^{\neg \epsilon}_{k}$.

\textbf{Construction part 1: $n=2k-1$ is odd.} 
By assumption, $\mathcal{F}_{k-1}$ agrees about $\partial (S_{k-1})$. Let $f$ denote the unique edge from $S_{k-1}$ to $S_{k} \setminus S_{k-1}$. We first apply Lemma~\ref{lem_increasingFnegEpsilon} to $\Fcal_{k-1}$ in order to find a thick $T$-tribe $\Fcal_{k}$ concentrated at $\epsilon$ satisfying properties \ref{itemagree}--\ref{consistentpushingalong}. 
In particular, $\Fcal_{k}$ agrees about $\partial (S_{k})$ and $S^{\neg \epsilon}_{k-1} \subseteq S^{\neg \epsilon}_{k}$

We first note that if $f \notin \partial_{\epsilon} (S_{k-1})$, then $S^{\neg \epsilon}_{k-1} = S^{\neg \epsilon}_{k}$, and we can simply take $Q^{n+1}_i := Q^{n}_i$ for all $i \in [k]$, $\mathcal{E}_{n+1} := \mathcal{E}_n$ and $X_{n+1} := X_n$. 

Otherwise, we have $f \in \partial_{\epsilon} (S_{k-1})$. By Lemma~\ref{lem_subtribesinheritconcentration}(2) $\Fcal_{k}$ is concentrated at $\epsilon$, and so we may pick a collection $\{H_1,\ldots,H_N\}$ of disjoint subdivisions of $T$ from some $F \in \mathcal{F}_{k}$, all of which are contained in $C(X_n ,\epsilon)$, where $N=|\mathcal{E}_n|$. 
By Lemma~\ref{l:weaklink} there is some linkage $\mathcal{P} \subseteq C(X_n ,\epsilon)$ from 

\[ 
    \mathcal{E}_n \, \text{ to } \; (H_j(R_f) \colon j \in [N] ), 
\]
which is after $X_n$. Let us suppose that the linkage $\mathcal{P}$ joins a vertex $x_{e,i} \in E^n_{e,i}$ to $y_{\sigma(e,i)} \in H_{\sigma(e,i)}(R_f)$ via a path $P_{e,i} \in \mathcal{P}$. 
Let $z_{\sigma(e,i)}$ be a vertex in $R_f$ such that $y_{\sigma(e,i)} \leq H_{\sigma(e,i)}(z_{\sigma(e,i)})$ in the tree order on $H_{\sigma(e,i)}(T)$.

By property \ref{consistentpushingalong} of $\Fcal_{k}$ in Lemma~\ref{lem_increasingFnegEpsilon}, we may assume without loss of generality that for each $H_j$ there is a another member $H'_j \subseteq H_j$ of $\Fcal_k$ such that $H'_j(T_{f^+}) \subseteq_r H_j(T_{z_j})$. Let $\hat{P}_j \subseteq H'_j$ denote the path from $H_j(y_j)$ to $H'_j(f^+)$. 

Now for each $i \in [k]$, define
\[
    Q^{n+1}_i = Q^n_i \cup E^n_{f,i} x_{f,i} P_{f,i} y_{\sigma(f,i)} \hat{P}_{\sigma(f,i)}  \cup H'_{\sigma(f,i)}(S^{\neg \epsilon}_{k} \setminus S^{\neg \epsilon}_{k-1}).
\]
By construction, each $Q^{n+1}_i$ is a $T$-suitable subdivision of $S^{\neg \epsilon}_{k}$.

By Lemma~\ref{lem_increasingFnegEpsilon}\ref{itemSeparation} we may find a finite set $X_{n+1} \subseteq \Gamma$ with $X_{n} \subseteq X_{n+1}$ such that 
\[
    C(X_{n+1},\epsilon)  \cap \big(\bigcup_{i \in [k]} Q^{n+1}_i \big) = \emptyset.
\]
This set $X_{n+1}$ will be our bounder.

Define an extender $\mathcal{E}_{n+1} = ( E^{n+1}_{e,i} \colon e \in \partial_\epsilon (S_{k}), i \in [k])$ for the $Q^{n+1}_i$ as follows:
\begin{itemize}
    \item For $e \in \partial_\epsilon (S_{k-1}) \setminus \{f\}$, let $ E^{n+1}_{e,i} := E^n_{e,i} x_{e,i} P_{e,i} y_{\sigma(e,i)} H_{\sigma(e,i)}(R_f)$.
    \item For $e \in \partial_\epsilon (S_{k}) \setminus \partial (S_{k-1})$, let $E^{n+1}_{e,i} := H'_{\sigma(e,i)}(R_e)$. 
\end{itemize}

Since each $H_{\sigma(e,i)},H'_{\sigma(e,i)} \in \bigcup \Fcal_k$, and $\Fcal_k$ determines that $R_f$ converges to $\epsilon$, these are indeed $\epsilon$ rays. 
Furthermore, since $H'_{\sigma(e,i)} \subseteq H_{\sigma(e,i)}$ and $\{H_1,\ldots,H_N\}$ are disjoint, it follows that the rays are disjoint.

\textbf{Construction part 2: $n=2k$ is even.} 
If $\partial_\epsilon (S_k) = \emptyset$, then $S^{\neg \epsilon}_k = S$, and so picking any element $Q^{n+1}_{k+1}$ from $\Fcal_k$ with $Q^{n+1}_{k+1} \subseteq C(X_n,\epsilon)$ gives us a further copy of $S$ disjoint from all the previous ones. 
Using Lemma~\ref{lem_increasingFnegEpsilon}(\ref{itemSeparation}), there is a suitable bounder $X_{n+1}\supseteq X_n$ for $Q^{n+1}_{k+1}$, and we are done. 
Otherwise, pick $e_0 \in \partial_\epsilon (S_n)$ arbitrary. 

Since $\Fcal_{k}$ is concentrated at $\epsilon$, we may pick a collection $\{H_1,\ldots,H_N\}$ of disjoint subdivisions of $T$ from $\mathcal{F}_{k}$ all contained in $C(X_n ,\epsilon)$, where $N$ is large enough so that we may apply Lemma~\ref{l:link} to find a linkage $\mathcal{P} \subseteq C(X_n ,\epsilon)$ from 
\[ 
    \mathcal{E}_n \, \text{ to } \; (H_i(R_{e_0}) \colon i \in [N] ), 
\]
after $X_n$, avoiding say $H_1$. 
Let us suppose the linkage $\mathcal{P}$ joins a vertex $x_{e,i} \in E^n_{e,i}$ to $y_{\sigma(e,i)} \in H_{\sigma(e,i)}(R_{e_0})$ via a path $P_{e,i} \in \mathcal{P}$. 
Define 
\[
    Q^{n+1}_{k+1} = H_1(S^{\neg \epsilon}_k).
\]
Note that $Q^{n+1}_{k+1}$ is a $T$-suitable subdivision of $S^{\neg \epsilon}_k$.
 
By Lemma~\ref{lem_increasingFnegEpsilon}(\ref{itemSeparation}) there is a finite set $X_{n+1} \subseteq \Gamma$ with $X_{n} \subseteq X_{n+1}$ such that 
$C(X_{n+1},\epsilon)  \cap Q^{n+1}_{k+1} = \emptyset$. 
This set $X_{n+1}$ will be our new bounder.

Define the extender $\mathcal{E}_{n+1} = ( E^{n+1}_{e,i} \colon e \in \partial_\epsilon (S_{k+1}), i \in [k+1] )$ of $\epsilon$-rays as follows:
\begin{itemize}
    \item For $i \in [k]$, let $ E^{n+1}_{e,i} := E^n_{e,i} x_{e,i} P_{e,i} y_{\sigma(e,i)} H_{\sigma(e,i)}(R_{e_0})$.
    \item For $i = k+1$, let $E^{n+1}_{e,k+1} := H_1(R_e)$ for all $e \in \partial_\epsilon (S_{k+1})$.
\end{itemize}

Once the construction is complete, let us define $H_i := \bigcup_{n \geq 2i-1} Q^n_i$.

Since $\bigcup_{n \in \N} S^{\neg \epsilon}_n =S$, and due to the extension property (2), the collection $(H_i)_{i \in \N}$ is an $S$-horde. \hfill 
\qed

We remark that our construction so far suffices to give a complete proof that countable trees are $\leq$-ubiquitous. 
Indeed, it is well-known that an end of $\Gamma$ is dominated by infinitely many distinct vertices if and only if $\Gamma$ contains a subdivision of $K_{\aleph_0}$ \cite[Exercise 19, Chapter 8]{D16}, in which case proving ubiquity becomes trivial:

\begin{lemma}
    For any countable graph $G$, we have $\aleph_0 \cdot G \subseteq K_{\aleph_0}$.
\end{lemma}
\begin{proof}
    By partitioning the vertex set of $K_{\aleph_0}$ into countably many infinite parts, we see that $\aleph_0 \cdot K_{\aleph_0} \subseteq K_{\aleph_0}$. 
    Also, clearly $G \subseteq K_{\aleph_0}$. Hence, we have $\aleph_0 \cdot G \subseteq \aleph_0 \cdot K_{\aleph_0} \subseteq K_{\aleph_0}$.
\end{proof}

\subsection{Infinitely many vertices dominate \texorpdfstring{$\epsilon$}{\textepsilon}}
\label{sec7.3}

The argument in this case is very similar to that in the previous subsection. We define bounders and extenders just as before. 
We once more assume inductively that for some $n \in \Nbb$, with $r:= \lfloor n/2 \rfloor$, we have objects given by (1)-(4) as in the last section, and which in addition satisfy
\begin{itemize}
    \item[(5)] $\mathcal{F}_{r}$ agrees about $V(S_{r})$.
    \item[(6)] For any $t \in V_{\epsilon}(S_{r})$ the vertex $Q_i^n(t)$ dominates $\epsilon$.
\end{itemize}

The base case is again trivial, so suppose that our construction has proceeded to step $n \geq 0$. 
The construction is split into two parts just as before, where the case $n=2k$, in which we need to refine our $T$-tribe and find a new copy $Q^{n+1}_{k+1}$ of $S^{\neg \epsilon}_{k}$, proceeds just as in the last section.

If $n = 2k-1$ is odd, and if $f \in \partial_{\neg \epsilon} (S_{k-1})$ or $\partial_{\epsilon}(S_{k-1})$, then we proceed as in the last subsection. But these are no longer the only possibilities. 
It follows from the definition of $S^{\neg \epsilon}_k$ that there is one more option, namely that $f^- \in V_{\epsilon}(S_{k})$. In this case we modify the steps of the construction as follows:

We first apply Lemma~\ref{lem_increasingFnegEpsilon} to $\Fcal_{k-1}$ in order to find a thick $T$-tribe $\Fcal_{k-1}$ which agrees about $\partial (S_k)$ and $V(S_k)$.

Then, by applying Lemma \ref{raystopaths} to tails of the rays $E_{e,i}^n$ in $C_{\Gamma}(X_n, \epsilon)$, we obtain a family $\Pcal_{n+1}$ of paths $P_{e,i}^{n+1}$ which are disjoint from each other and from the $Q_i^n$ except at their initial vertices, where the initial vertex of $P_{e,i}^{n+1}$ is $Q^n_i(e^-)$ and the final vertex $y_{e,i}^{n+1}$ of $P_{e,i}^{n+1}$ dominates $\epsilon$.

Since $\Fcal_{k}$ is concentrated at $\epsilon$, we may pick a collection $\{H_1,\ldots,H_k\}$ of disjoint $T$-minors from $\mathcal{F}_{k}$ all contained in $C(X_n \cup \bigcup \Pcal_{n+1},\epsilon)$. 

Now for each $i \in [k]$, define  
\[
    \hat{Q}^{n+1}_i = Q^n_i \cup H_i(f^-) \cup H_i(S^{\neg \epsilon}_{k} \setminus S^{\neg \epsilon}_{k-1}).
\]
These are almost $T$-suitable subdivisions of $S^{\neg \epsilon}_{k}$, except we need to add a path between $Q^n_i(f^-)$ and $H_i(f^-)$.


By applying Lemma \ref{raystopaths} to tails of the rays $H_i(R_e)$ inside $C(X_n \cup \bigcup \Pcal_{n+1},\epsilon)$ with $e \in \partial_{\epsilon} (S_{k+1}) \setminus \partial (S_k)$ we can construct a family $\Pcal'_{n+1} := \{P_{e,i}^{n+1} \colon e \in \partial_{\epsilon} (S_{k+1}) \setminus \partial_{\epsilon} (S_k), i \leq k\}$ of paths which are disjoint from each other and from the $\hat{Q}_i^{n+1}$ except at their initial vertices, where the initial vertex of $P_{e,i}^{n+1}$ is $H_i(e^-)$ and the final vertex $y_{e,i}^{n+1}$ of $P_{e,i}^{n+1}$ dominates $\epsilon$. 
Therefore the family 
\[
    \Pcal_{n+1} \cup \Pcal'_{n+1} = ( P_{e,i}^{n+1} \colon e \in \partial_{\epsilon} (S_{k+1}), i \in [k])
\]
is a family of disjoint paths, which are also disjoint from the $\hat{Q}_i^{n+1}$ except at their initial vertices, where the initial vertex of $P_{e,i}^{n+1}$ is $H_i(e^-)$ or $Q^n_i(e^-)$ and the final vertex $y_{e,i}^{n+1}$ of $P_{e,i}^{n+1}$ dominates $\epsilon$.

Since $Q_i^n(f^-)$ and $H_i(f^-)$ both dominate $\epsilon$ for all $i$, we may recursively build a sequence $\hat{\mathcal P}_{n+1} = \{ \hat{P}_i \colon 1 \leq i \leq k\}$ of disjoint paths $\hat{P}_i$ from $Q_i^n(f^-)$ to $H_i(f^-)$ with all internal vertices in $C( X_{n+1} \cup \left(\bigcup {\mathcal P}'_{n+1} \cup \bigcup \Pcal_{n+1}\right),\epsilon)$. 
Letting $Q^{n+1}_i = \hat{Q}_i^{n+1} \cup \hat{P}_i$, we see that each $Q^{n+1}_i$ is a $T$-suitable subdivision of $S^{\neg \epsilon}_{k}$ in $\Gamma$.

Our new bounder will be $X_{n+1}:=X_{n} \cup \bigcup \hat{\mathcal P}_{n+1} \cup \bigcup \Pcal'_{n+1} \cup \bigcup \Pcal_{n+1}$.

Finally, let us apply Lemma \ref{pathstorays} to $Y := \{y_{e,i}^{n+1} \colon e \in \partial_{\epsilon} (S_{n+1}), i \leq k\}$ in $\Gamma[Y \cup C( X_{n+1}, \epsilon) ]$. 
This gives us a family of disjoint rays
\[
    \hat{\mathcal{E}}_{n+1} = ( \hat{E}^{n+1}_{e,i} \colon e \in \partial_{\epsilon} (S_{k+1}), i \in [k] )
\]
such that $\hat{E}^{n+1}_{e,i}$ has initial vertex $y_{e,i}^{n+1}$. Let us define our new extender $\mathcal{E}_{n+1}$ given by
\begin{itemize}
    \item $E^{n+1}_{e,i} = Q^n_i(e^-) P^{n+1}_{e,i} y_{e,i}^{n+1} \hat{E}^{n+1}_{e,i}$ if  $e \in \partial_{\epsilon} (S_k), i \in [k]$;
    \item $E^{n+1}_{e,i} = H_i(e^-) P^{n+1}_{e,i} y_{e,i}^{n+1} \hat{E}^{n+1}_{e,i}$ if  $e \in \partial_{\epsilon} (S_{k+1}) \setminus \partial (S_k), i \in [k]$.
\end{itemize}

This concludes the proof of Theorem~\ref{t:countembed}. \hfill \qed

\section{The induction argument}
\label{sec:Induction}

We consider $T$ as a rooted tree with root $r$. 
In Section~\ref{sec:countable-subtrees} we constructed an $S$-horde for any countable subtree $S$ of $T$. 
In this section we will extend an $S$-horde for some specific countable subtree $S$ to a $T$-horde, completing the proof of Theorem~\ref{t:tree}.

Recall that for a vertex $t$ of $T$ and an infinite cardinal $\kappa$ we say that a child $t'$ of $t$ is $\kappa$-embeddable if there are at least $\kappa$ children $t''$ of $t$ such that $T_{t'}$ is a (rooted) topological minor of $T_{t''}$ (Definition~\ref{def_kappachildren}). 
By Corollary~\ref{lem_rotatearoundKappa}, the number of children of $t$ which are not $\kappa$-embeddable is less than $\kappa$. 

\begin{definition}[$\kappa$-closure]
    Let $T$ be an infinite tree with root $r$. 
    \begin{itemize}
        \item If $S$ is a subtree of $T$ and $S'$ is a subtree of $S$, then we say that $S'$ is \emph{$\kappa$-closed in $S$} if for any vertex $t$ of $S'$ all children of $t$ in $S$ are either in $S'$ or else are $\kappa$-embeddable.
    \item The \emph{$\kappa$-closure of $S'$ in $S$} is the smallest $\kappa$-closed subtree of $S$ including $S'$.
    \end{itemize}
\end{definition}

\begin{lemma}
    \label{lem_Kappaclosure}
	Let $S'$ be a subtree of $S$. If $\kappa$ is a uncountable regular cardinal and $S'$ has size less than $\kappa$, then the $\kappa$-closure of $S'$ in $S$ also has size less than $\kappa$.
\end{lemma}

\begin{proof}
    Let $S'(0) := S'$ and define inductively $S'(n+1)$ to consist of $S'(n)$ together with all non-$\kappa$-embeddable children contained in $S$ for all vertices of $S'(n)$. 
    It is clear that $\bigcup_{n \in \N} S'(n)$ is the $\kappa$-closure of $S'$. 
    If $\kappa_n$ denotes the size of $S'(n)$, then $\kappa_{n} < \kappa$ by induction with Corollary~\ref{lem_rotatearoundKappa}. Therefore, the size of the $\kappa$-closure is bounded by $ \sum_{n \in \N} \kappa_n < \kappa$, since $\kappa$ has uncountable cofinality.
\end{proof}

We will construct the desired $T$-horde via transfinite induction on the cardinals ${\mu \leq \lvert T \rvert}$. 
Our first lemma illustrates the induction step for regular cardinals.

\begin{lemma}\label{1step}
    Let $\kappa$ be an uncountable regular cardinal. Let $S$ be a rooted subtree of $T$ of size at most $\kappa$ and let $S'$ be a $\kappa$-closed rooted subtree of $S$ of size less than $\kappa$. Then any $S'$-horde $(H_i \colon i \in \N)$ can be extended to an $S$-horde.
\end{lemma}
\begin{proof}
    Let $(s_{\alpha} \colon \alpha < \kappa)$ be an enumeration of the vertices of $S$ such that the parent of any vertex appears before that vertex in the enumeration, and for any $\alpha$ let $S_{\alpha}$ be the subtree of $T$ with vertex set $V(S') \cup \{s_{\beta} \colon \beta < \alpha\}$. Let $\bar S_{\alpha}$ denote the $\kappa$-closure of $S_{\alpha}$ in $S$, and observe that $| \bar S_{\alpha}| < \kappa$ by Lemma~\ref{lem_Kappaclosure}.

    We will recursively construct for each $\alpha$ an $\bar S_{\alpha}$-horde $(H^{\alpha}_i \colon i \in \N)$ in $\Gamma$, where each of these hordes extends all the previous ones. 
    For $\alpha = 0$ we let $H^0_i = H_i$ for each $i \in \mathbb{N}$. 
    For any limit ordinal $\lambda$ we have $\bar S_{\lambda} = \bigcup_{\beta < \lambda} \bar S_{\beta}$, and so we can take $H^{\lambda}_i = \bigcup_{\beta < \lambda}H^{\beta}_i$ for each $i \in \mathbb{N}$.

    For any successor ordinal $\alpha = \beta + 1$, if $s_{\beta} \in \bar S_{\beta}$, then $\bar S_{\alpha} = \bar S_{\beta}$, and so we can take $H^{\alpha}_i = H^{\beta}_i$ for each $i \in \mathbb{N}$. 
    Otherwise, $\bar S_{\alpha}$ is the $\kappa$-closure of $\bar S_{\beta} + s_{\beta}$, and so $\bar S_{\alpha} - \bar S_{\beta}$ is a subtree of $T_{s_{\beta}}$. 
    Furthermore, since $s_{\beta}$ is not contained in $\bar S_{\beta}$, it must be $\kappa$-embeddable. 

    Let $s$ be the parent of $s_{\beta}$. 
    By suitability of the $H^{\beta}_i$, we can find for each $i \in \mathbb{N}$ some subdivision $\hat H_i$ of $T_s$ with $\hat H_i(s) = H^{\beta}_i(s)$. 
    We now build the $H^{\alpha}_i$ recursively in $i$ as follows: 

    Let $t_i$ be a child of $s$ such that $T_{t_i}$ has a rooted subdivision $K$ of $T_{s_{\beta}}$, and such that ${\hat H_i(T_{t_i} + s) - \hat H_i(s)}$ is disjoint from all $H^{\alpha}_j$ with $j < i$ and from all $H^{\beta}_j$. Since there are $\kappa$ disjoint  possibilities for $K$, and all $H^{\alpha}_j$ with $j < i$ and all $H^{\beta}_j$ cover less than $\kappa$ vertices in $\Gamma$, such a choice of $K$ is always possible.
    Then let $H^{\alpha}_i$ be the union of $H^{\beta}_i$ with $\hat H_i(K(\bar S_{\alpha} - \bar S_{\beta}) + st_i)$.

    This completes the construction of the $(H^{\alpha}_i \colon i \in \N)$. 
    Obviously, each $H^{\alpha}_i$ for $i \in \N$ is a subdivision of $\bar S_\alpha$ with $H^{\alpha}_i(\bar S_\gamma) = H^{\gamma}_i$ for all $\gamma < \alpha$, and all of them are pairwise disjoint for $i \neq j \in \N$.
    Moreover, $H^{\alpha}_i$ is $T$-suitable since for all vertices~$H^{\alpha}_i(t)$ whose $t$-suitability is not witnessed in previous construction steps, their suitability is witnessed now by the corresponding subtree of~$\hat H_i$.
    Hence $(\bigcup_{\alpha < \kappa} H^{\alpha}_i \colon i \in \N)$ is the desired $S$-horde extending $(H_i \colon i \in \N)$.
\end{proof}

Our final lemma will deal with the induction step for singular cardinals. 
The crucial ingredient will be to represent a tree $S$ of singular cardinality $\mu$ as a continuous increasing union of ${<}\mu$-sized subtrees $(S_\rho \colon \rho < \cf(\mu) )$ where each $S_\rho$ is $|S_\rho|^+$-closed in $S$. 
This type of argument is based on Shelah's singular compactness theorem, see e.g.\ \cite{shelah1975compactness}, but can be read without knowledge of the  paper.

\begin{definition}[$S$-representation]
    For a tree $S$ with $|S|=\mu$, we call a sequence $\mathcal{S}=(S_\rho \colon \rho < \cf(\mu) )$ of subtrees of $S$ with $|S_\rho| = \mu_\rho$ an \emph{$S$-representation} if
    \begin{itemize}
        \item $(\mu_{\rho} \colon \rho < \cf(\mu))$ is a strictly increasing continuous sequence of cardinals less than $\mu$ which is cofinal for $\mu$, 
        \item $S_\rho \subseteq S_{\rho'}$ for all $\rho < \rho'$, i.e.\ $\mathcal{S}$ is \emph{increasing},
        \item for every limit $\lambda < \cf(\mu)$ we have $\bigcup_{\rho < \lambda} S_\rho = S_\lambda$, i.e.\ $\mathcal{S}$ is \emph{continuous},
        \item $\bigcup_{\rho < \cf(\mu)} S_\rho = S$, i.e.\ $\mathcal{S}$ is \emph{exhausting},
        \item $S_\rho$ is $\mu_\rho^+$-closed in $S$ for all $\rho< \cf(\mu)$, where $\mu_\rho^+$ is the successor cardinal of $\mu_{\rho}$. 
    \end{itemize}
    Moreover, for a tree $S' \subseteq S$ we say that $\mathcal{S}$ is an \emph{$S$-representation extending $S'$} if additionally
    \begin{itemize}
        \item $S' \subseteq S_\rho$ for all $\rho < \cf(\mu)$.
    \end{itemize}
\end{definition}

\begin{lemma}
    \label{lem_nicerep}
    For every tree $S$ of singular cardinality and every subtree $S'$ of $S$ with $|S'|< |S|$ there is an $S$-representation extending $S'$.
\end{lemma}

\begin{proof}
    Let $|S|=\mu$ be singular, and let $|S'|=\kappa$. Let $(s_{\alpha} \colon \alpha < \mu)$ be an enumeration of the vertices of~$S$. 
    Let $\gamma$ be the cofinality of $\mu$ and let $(\mu_{\rho} \colon \rho < \gamma)$ be a strictly increasing continuous cofinal sequence of cardinals less than $\mu$ with $\mu_0 > \gamma$ and $\mu_0 > \kappa$. 
    By recursion on $i$ we choose for each $i \in \Nbb$ a sequence $(S^i_{\rho} \colon \rho < \gamma)$ of subtrees of $S$ of cardinality $\mu_\rho$, where the vertices of each $S^i_{\rho}$ are enumerated as $(s^i_{\rho, \alpha} \colon \alpha < \mu_{\rho})$, such that:
    
    \begin{enumerate}
        \item\label{nice2} $S^i_{\rho}$ is $\mu_{\rho}^+$-closed.
        \item\label{nice3} $S'$ is a subtree of $S^i_{\rho}$.
        \item\label{nice4} $S^i_{\rho'}$ is a subtree of $S^i_{\rho}$ for $\rho' < \rho$. 
        \item\label{nice5} $s_{\alpha} \in S^{i}_{\rho}$ for $\alpha < \mu_{\rho}$. 
        \item\label{nice6} $s^j_{\rho', \alpha} \in S^i_{\rho}$ for any $j < i$, $\rho \leq \rho' < \gamma$ and $\alpha < \mu_{\rho}$
    \end{enumerate}
    
    This is achieved by recursion on $\rho$ as follows: 
    For any given $\rho < \gamma$, let $X^i_{\rho}$ be the set of all vertices which are forced to lie in $S^i_{\rho}$ by conditions \ref{nice3}--\ref{nice6}, that is, all vertices of $S'$ or of $S^i_{\rho'}$ with $\rho' < \rho$, all $s_\beta$ with $\beta < \mu_\rho$ and all $s^j_{\rho', \alpha}$ with $j < i$, $\rho \leq \rho' < \gamma$ and $\alpha < \mu_{\rho}$. 
    Then $X^i_{\rho}$ has cardinality $\mu_{\rho}$ and so it is included in a subtree of $S$ of cardinality~$\mu_{\rho}$. 
    We take $S^i_{\rho}$ to be the $\mu_{\rho}^+$-closure of this subtree in $S$. Note that, since $\mu_{\rho}^+$ is regular, it follows from Lemma~\ref{lem_Kappaclosure} that $S^i_{\rho}$ has cardinality $\mu_{\rho}$.
    
    For each $\rho < \gamma$, let $S_{\rho} := \bigcup_{i \in \Nbb} S_{\rho}^i$. 
    Then each $S_{\rho}$ is a union of $\mu_{\rho}^+$-closed trees and so is $\mu_{\rho}^+$-closed itself. 
    Furthermore, each $S_\rho$ clearly has cardinality $\mu_{\rho}$. 
    
    It follows from \ref{nice5} that $S = \bigcup_{\rho < \gamma} S_\rho$. Thus, it remains to argue that our sequence is indeed continuous, i.e.\ that for any limit ordinal $\lambda < \gamma$ we have $S_{\lambda} = \bigcup_{\rho < \lambda}S_{\rho}$. 
    The inclusion $\bigcup_{\rho < \lambda}S_{\rho} \subseteq S_{\lambda}$ is clear from \ref{nice4}. 
    For the other inclusion, let $s$ be any element of~$S_\lambda$. 
    Then there is some $i \in \Nbb$ with $s \in S^i_{\lambda}$ and so there is some $\alpha < \mu_{\alpha}$ with $s = s^i_{\lambda, \alpha}$. 
    Then by continuity there is some $\sigma < \lambda$ with $\alpha < \mu_{\sigma}$ and so $s \in S^{i+1}_{\sigma} \subseteq S_{\sigma} \subseteq \bigcup_{\rho < \lambda} S_{\rho}$.
\end{proof}

\begin{lemma}\label{allsteps}
    Let $\mu$ be a cardinal. Then for any rooted subtree $S$ of $T$ of size $\mu$ and any uncountable regular cardinal $\kappa \leq \mu$, any $S'$-horde $(H_i \colon i \in \N)$ of a $\kappa$-closed rooted subtree $S'$ of $S$ of size less than $\kappa$ can be extended to an $S$-horde.
\end{lemma}
\begin{proof}
    The proof is by transfinite induction on $\mu$. 
    If $\mu$ is regular, we let $S''$ be the $\mu$-closure of $S'$ in $S$. 
    Thus $S''$ has size less than $\mu$. 
    So by the induction hypothesis $(H_i \colon i \in \N)$ can be extended to an $S''$-horde, which by Lemma \ref{1step} can be further extended to an $S$-horde.

  	So let us assume that $\mu$ is singular, and write $\gamma = \cf(\mu)$. By Lemma~\ref{lem_nicerep}, fix an $S$-representation $\mathcal{S}=(S_\rho \colon \rho < \cf(\mu) )$ extending $S'$ with $|S'|<|S_0|$. 

We now recursively construct for each $\rho < \gamma$ an $S_{\rho}$-horde $(H^{\rho}_i \colon i \in \N)$, where each of these hordes extends all the previous ones and $(H_i \colon i \in \N)$. 
   Using that each $S_\rho$ is $\mu_\rho^+$-closed in $S$, we can find $(H^0_i \colon i \in \N)$ by the induction hypothesis, and if $\rho$ is a successor ordinal we can find $(H^{\rho}_i \colon i \in \N)$ by again using the induction hypothesis. 
    For any limit ordinal $\lambda$ we set $H^{\lambda}_i = \bigcup_{\rho < \lambda} H^{\rho}_i$ for each $i \in \mathbb{N}$, which yields an $S_\lambda$-horde by the continuity of $\mathcal{S}$. 
    
    This completes the construction of the $H^{\rho}_i$. 
    Then $(\bigcup_{\rho < \gamma} H^{\rho}_i \colon i \in \N)$ is an $S$-horde extending $(H_i \colon i \in \N)$.
\end{proof}

Finally, with the right induction start we obtain the following theorem and hence a proof of Theorem~\ref{t:tree}.

\begin{theorem}
    Let $T$ be a tree and $\Gamma$ a graph such that $nT \leq \Gamma$ for every $n \in \Nbb$. 
    Then there is a $T$-horde, and hence $\aleph_0 T \leq \Gamma$.
\end{theorem}
\begin{proof}
    By Theorem~\ref{t:countembed}, we may assume that $T$ is uncountable. Let $S'$ be the $\aleph_1$-closure of the root $\{r\}$ in $T$. 
    Then $S'$ is countable by Lemma~\ref{lem_Kappaclosure} and so there is an $S'$-horde in $\Gamma$ by Theorem~\ref{t:countembed}. 
    This can be extended to a $T$-horde in $\Gamma$ by Lemma \ref{allsteps} with $\mu = |T|$.
\end{proof}

\bibliographystyle{plain}
\bibliography{ubiq}

\begin{thebibliography}{10}

\bibitem{A77}
T.~Andreae.
\newblock {B}emerkung zu einem {P}roblem aus der {T}heorie der unendlichen
  {G}raphen.
\newblock {\em Abhandlungen aus dem Mathematischen Seminar der Universit{\"a}t
  Hamburg}, 46(1):91, 1977.

\bibitem{A79}
T.~Andreae.
\newblock {\"U}ber eine {E}igenschaft lokalfiniter, unendlicher {B}\"{a}ume.
\newblock {\em Journal of Combinatorial Theory, Series B}, 27(2):202--215,
  1979.

\bibitem{A02}
T.~Andreae.
\newblock On disjoint configurations in infinite graphs.
\newblock {\em Journal of Graph Theory}, 39(4):222--229, 2002.

\bibitem{A13}
T.~Andreae.
\newblock Classes of locally finite ubiquitous graphs.
\newblock {\em Journal of Combinatorial Theory, Series B}, 103(2):274--290,
  2013.

\bibitem{BEEGHPTII}
N.~Bowler, C.~Elbracht, J.~Erde, P.~Gollin, K.~Heuer, M.~Pitz, and M.~Teegen.
\newblock Ubiquity in graphs {II}: Ubiquity of graphs with non-linear end
  structure.
\newblock in preparation.

\bibitem{BEEGHPTIII}
N.~Bowler, C.~Elbracht, J.~Erde, P.~Gollin, K.~Heuer, M.~Pitz, and M.~Teegen.
\newblock Ubiquity in graphs {III}: Ubiquity of a class of locally finite
  graphs.
\newblock in preparation.

\bibitem{BEEGHPTIV}
N.~Bowler, C.~Elbracht, J.~Erde, P.~Gollin, K.~Heuer, M.~Pitz, and M.~Teegen.
\newblock Ubiquity in graphs {IV}: Ubiquity of graphs of bounded tree-width.
\newblock in preparation.

\bibitem{D16}
R.~Diestel.
\newblock {\em {Graph Theory}}.
\newblock Springer, 5th edition, 2016.

\bibitem{Ends}
R.~Diestel and D.~K\"uhn.
\newblock {Graph-theoretical versus topological ends of graphs}.
\newblock {\em Journal of Combinatorial Theory, Series B}, 87:197--206, 2003.

\bibitem{H65}
R.~Halin.
\newblock {\"U}ber die {M}aximalzahl fremder unendlicher {W}ege in {G}raphen.
\newblock {\em Mathematische Nachrichten}, 30(1-2):63--85, 1965.

\bibitem{H70}
R.~Halin.
\newblock {D}ie {M}aximalzahl fremder zweiseitig unendlicher {W}ege in
  {G}raphen.
\newblock {\em Mathematische Nachrichten}, 44(1-6):119--127, 1970.

\bibitem{halin1975problem}
R~Halin.
\newblock A problem in infinite graph-theory.
\newblock {\em Abhandlungen aus dem Mathematischen Seminar der Universit{\"a}t
  Hamburg}, 43(1):79--84, 1975.

\bibitem{L76}
J.~Lake.
\newblock A problem concerning infinite graphs.
\newblock {\em Discrete Mathematics}, 14(4):343--345, 1976.

\bibitem{laver1978better}
R.~Laver.
\newblock Better-quasi-orderings and a class of trees.
\newblock {\em Studies in foundations and combinatorics}, 1:31--48, 1978.

\bibitem{nash1965well}
C.~St. J.~A. Nash-Williams.
\newblock On well-quasi-ordering infinite trees.
\newblock {\em Mathematical proceedings of the Cambridge philosophical
  society}, 61(3):697--720, 1965.

\bibitem{shelah1975compactness}
S.~Shelah.
\newblock A compactness theorem for singular cardinals, free algebras,
  whitehead problem and tranversals.
\newblock {\em Israel Journal of Mathematics}, 21(4):319--349, 1975.

\bibitem{T88}
R.~Thomas.
\newblock A counter-example to `{W}agner's conjecture' for infinite graphs.
\newblock {\em Mathematical Proceedings of the Cambridge Philosophical
  Society}, 103(1):55--57, 1988.

\bibitem{W76}
D.~R. Woodall.
\newblock A note on a problem of {H}alin's.
\newblock {\em Journal of Combinatorial Theory, Series B}, 21(2):132--134,
  1976.

\end{thebibliography}

\end{document}